\documentclass[a4paper]{amsart}
 
 \usepackage{url} 
 \usepackage{framed}
\usepackage[T1]{fontenc}
\usepackage{amsmath}
\usepackage{amssymb}
\usepackage[usenames]{color}
\usepackage{amssymb}
\usepackage{latexsym}
 \usepackage{graphicx}
\usepackage{psfrag}

\newtheorem{definition}{Definition}[section] 
\newtheorem{theorem}[definition]{Theorem} 
\newtheorem{lemma}[definition]{Lemma} 
 
\newtheorem{proposition}[definition]{Proposition} 
\newtheorem{corollary}[definition]{Corollary}

\usepackage{graphicx}
\usepackage{float}
\usepackage{hyperref}
\usepackage{epstopdf}
\usepackage{caption}
\usepackage{subcaption}
\usepackage[T1]{fontenc}
\input xypic
\usepackage{xcolor}
\usepackage{amsmath}
\usepackage{amssymb, graphics, color}
\usepackage{multicol}
\usepackage{mathrsfs}
\usepackage[xcolor]{mdframed}
\usepackage{lipsum}
\usepackage{dsfont} 
\definecolor{pasgre}{rgb}{0.67, 0.87, 0.67}
\definecolor{inchworm}{rgb}{0.7, 0.93, 0.36}
\definecolor{aqua}{rgb}{0.0, 1.0, 1.0}
\definecolor{myred}{RGB}{169, 50, 38}
\definecolor{myblue}{RGB}{41,128,185}
\definecolor{mygreen}{RGB}{169,223,191}
\definecolor{mygray}{RGB}{93,223,191}
\definecolor{myyellow}{RGB}{247,220,111}
\definecolor{mywhite}{RGB}{240,243,244}
\usepackage{tikz}

\def\hat{\widehat}

\def\bra{\langle}
\def\cet{\rangle}

\def\linspan{\hbox{span}}

\def\zeta{\mathbb{Z}_+}

\DeclareMathOperator*{\argmin}{arg\,min}

\def \p {\partial}
\def \a {\alpha}

\def \ba {\begin {eqnarray*} }	
\def \ea {\end {eqnarray*} }
\def \beq {\begin {equation}}
\def \eeq {\end {equation}}
\def \mbeq {\begin {eqnarray}}
\def \meeq {\end {eqnarray}}
\def \bfo {\begin {displaymath} }
\def \efo {\end {displaymath} }
\def \beq {\begin {eqnarray}}
\def \eeq {\end {eqnarray}}
\def \ba {\begin {eqnarray*}}
\def \ea {\end {eqnarray*}}

\def\tilde{\widetilde}

\def\M{{M}}
\def\M{{\cM}}

\def \HL {\Lambda}

\def \Z {{\mathbb{Z}}}

\def \R {{\mathbb{R}}}

\newcommand\pef[1]{(\ref{#1})}

\def \H2s {H^{s+1}_0(\partial \M\times [0,T/2])}

\def \supp {\hbox{supp }}
\def \diam {\hbox{diam }}

\def \det {\hbox{det}}
\def\bra{\langle}
\def\cet{\rangle}
\def \ve {\varepsilon}

\def \a {\alpha}

\def \A {{\mathcal A}}
\def \F {{\mathcal F}}

\def \M {{M}}
\def \F {{\mathcal F }}

\def \pat {\partial _t}
\def \pa0 {\partial _0}

\def \p {\partial}

\def \1 {{\mathds{1}}}

\def \HL {\Lambda}

\def \tilde{\widetilde}

\newcommand{\sijoitus}[2]%
{\operatornamewithlimits{\Bigl/}_{\!\!\!#1}^{\,#2}}

\newcommand{\vc}[2]{\begin{pmatrix} #1 \\ #2 \end{pmatrix}}

\newcommand{\vol}[0]{\operatorname{Vol}}

\newcommand{\newt}[1]{#1}

\newcommand{\newtextttt}[1]{#1}
\def\newtexttt{}
\def\newtextt{}
\def\newtext {}

\begin{document}
\title[Artificial point sources]{Construction of artificial point sources for a linear wave equation in unknown medium}
\date{} 
\author[Kirpichnikova, Korpela, Lassas, Oksanen] {Anna Kirpichnikova, Jussi Korpela, Matti Lassas, and Lauri Oksanen}
\address{
Anna Kirpichnikova, University of Stirling;
Jussi Korpela, University of Helsinki;
Matti Lassas,  University of Helsinki;
Lauri Oksanen, UCL.}
%
%







\maketitle


{\bf Abstract:} 
{\it  
  We study the wave equation on a bounded domain of $\R^m$ and  on a compact Riemannian manifold $M$ with boundary.  We assume that the coefficients of the wave equation are unknown but that we are given the hyperbolic Neumann-to-Dirichlet map $\Lambda$  that corresponds  to the physical measurements on the boundary.  
  Using the knowledge of $\Lambda$ we  construct a sequence of Neumann boundary values so that at a time $T$ the corresponding waves converge to zero while the time derivative of the waves converge to a
 delta distribution. Such waves are called an artificial point source. The convergence of the wave takes place in the function spaces naturally
 related to the energy of the wave. We apply the results for inverse problems and demonstrate the focusing of the waves numerically in the 1-dimensional case.
 }



 

\section{Introduction}
We  consider the wave equation in $M$ that is a bounded domain of
$\R^m,\,m\ge 1,$ or a compact manifold. Let $u=u^f(x,t)$ be the solution of the wave equation
\begin{align}
\label{eq: Wave}
\begin{cases}
  \p_t^2 u(x,t)+\A u(x,t)=0,\quad \hbox{ in } M\times \R_+,\\
  u|_{t=0}=0,\quad \p_t u|_{t=0}=0,  \\
  \p_\nu u|_{\p  M \times \R_+}=f, 
\end{cases}
\end{align}
where $\A$ is a selfadjoint  second order elliptic differential operator of
the form $$\A=-\Delta_g+\sum_{j=1}^mV_j(x)\p_{x_j},$$ where $\Delta_g$ is the Laplace operator associated to a Riemannian metric $g$, see \eqref{A} for precise definition.
Moreover,   $f\in L^2(\p  M\times \R_+)$ is a Neumann boundary value that physically corresponds to a boundary source, $u=u^f$ is the unique solution wave corresponding to the boundary source $f,$  and $\nu$ is the interior pointing normal vector of the boundary $\p M.$ We assume that we are given the \emph{Neumann-to-Dirichlet map}, $\Lambda f =u^f|_{\p M\times\R_+}$. The map $\Lambda$ corresponds to the knowledge of  measurements made on the boundary of the domain and it models the response $u^f|_{\p M\times\R_+}$
of the medium to a  source $f$ put on the boundary of $M$.


We show that using $\Lambda$ we can find a sequence of Neumann boundary values  $f_i$ such that
 the wave and its time derivative at the large enough time $T$, that is, the pair $
(u^{f_i}(\,\cdotp,T), u^{f_i}_t(\,\cdotp,T))$ converge in the energy norm to 
$(0,\frac 1{\hbox{\tiny Vol}(V)}\mathds{1}_V)$, as $i\to \infty$.
Here, $V(x)$ is the indicator function a small neighborhood of a point $\hat x\in M$ and  $\hbox{Vol}(V)$  is the Riemannian volume of $V$  in $(M,g)$.
More precisely, $\hat x=\gamma_{\hat z,\nu}({{\widehat t}})$ is a point on the normal geodesics emanating
from a boundary point $\hat z$.
%
%
%
%
Furthermore, when the neighborhood $V$  converges to 
the point $\hat x$, the limits $(0,\frac 1{\hbox{\tiny Vol}(V)}\mathds{1}_V)$
%
 converges  in
suitable function space to $(0,\delta_{\widehat{x}})$,
where $\delta_{\widehat{x}}$ is the Dirac delta distribution. We call the waves $u^{f_i}$ that concentrate their energy in a small neighbourhood $V$ of a point inside the domain the  \emph{focusing waves}. When $V\to \{\hat x\}$,  the  waves $u^{f_i}(x,t)$ converge in the set $M\times (T,\infty)$
 converge to  $G(x,t; \hat x, T)$, where 
 Green's function $G(x,t; x_0, t_0)$ that is the solution of
  \begin{equation}\label{GreensWaveEq}
\begin{cases}
  \left(\partial^2_t-\mathcal A\right) G(x, t;x_0,t_0)=\delta_{x_0}(x)\delta_{t_0}(t)\quad \hbox{on }M\times \R\\
  G(\cdot, \cdot; x_0, t_0)|_{t<t_0}=0;\,\,\partial_\nu G(\cdot, \cdot; x_0, t_0)|_{\partial M\times \R}=0.
\end{cases}
\end{equation}
Roughly speaking, the waves $u^{f_i}$  in the set $M\times (T,\infty)$
converge to the wave that is produced at a point source located at $( \hat x, T)$.
Due to this, we say that when $V\to \{x\}$, the limit of the focusing waves produces an \emph{artificial point source} at time $t=T$.

We emphasize that the boundary sources $f_i$ that produce focusing waves can be determined without
knowing  the coefficients of the operator $\A$,
that is, when the medium in $M$ is unknown and it is enough only to know the map $\Lambda$ that corresponds
to measurements done on the boundary of the domain.
Our main resut is the following:

\begin{theorem}
\label{Thm 1} Let $T>\frac 12\diam(M)$ and
 $\widehat x={\gamma}_{\widehat z,\nu}({{{{\widehat t}}}})\in M$,  $\widehat z\in {\p M}$,
$0<{{{{\widehat t}}}}<T$. Let  $\tau_{\p M}(\hat z)$ be 
the critical distance along the normal geodesic $\gamma_{\hat z,\nu}$,
defined in \eqref{critical distance}.

Then $(\p M,g|_{\p M})$ and the Neumann-to-Dirichlet map $\Lambda$ determine  Neumann boundary values $f_n{{(\alpha,\beta}},k)$, $n,k\in \Z_+,$ $\a,\beta>0$,
such that the following is true:

If ${{{{{{\widehat t}}}}}}<\tau_{\p M}(\widehat z)$ then  
\beq
\label{eq: A main limit}
 \lim_{\alpha\to 0^+}\lim_{\beta\to 0^+}\lim_{n\to \infty}
\vc{  u^{f_n{{(\alpha,\beta}},k)}(\cdot,T)}{\p_t u^{f_n{{(\alpha,\beta}},k)}(\cdot,T)}  =
\vc{0}{\frac{1}{\hbox{\tiny Vol}(\Omega_k)}\mathds{1}_{\Omega_k}}\quad\hbox{in  $H_0^1(M)\times L^2(M)$},\hspace{-15mm}
\eeq 
 where $\Omega_k\subset M$ are neighborhoods of $\hat x$ satisfying $\lim_{k\to\infty}\Omega_k=\{\hat x\}$.
Moreover,
\beq
\label{eq: A0}
 \lim_{k\to \infty} \left( \lim_{\alpha\to 0^+}\lim_{\beta\to 0^+}\lim_{n\to \infty}
\vc{  u^{f_n{{(\alpha,\beta}},k)}(\cdot,T)}{\p_t u^{f_n{{(\alpha,\beta}},k)}(\cdot,T)} \right) =
\vc{0}{ \delta_{\widehat{x}}},
\eeq 
where the inner limits with respect to $n, \beta, \alpha$ are in $H_0^1(M)\times L^2(M)$ and the outer limit with respect to $k$ is in the space $H^{-s+1}(M)\times H^{-s}(M)$ with  $s>{\dim(M)}/{2}$.
In addition, for $t>T$
\beq
\label{eq: A0 Green}
 \lim_{k\to \infty} \left(\lim_{\alpha\to 0^+}\lim_{\beta\to 0^+}\lim_{n\to \infty}
 u^{f_n{{(\alpha,\beta}},k)}(\,\cdot,t)\right) =G(\,\cdot,t;\hat x,T)
\eeq 
where the inner limits with respect to $n, \beta, \alpha$ are in $H_0^1(M)$ and the outer limit with respect to $k$ is in the space $H^{-s+1}(M)$.
  \newt{

If  ${{{{{{\widehat t}}}}}}>\tau_{\p M}(\widehat z)$, then limits \eqref{eq: A main limit}, \eqref{eq: A0}
and \eqref{eq: A0 Green} are equal to zero. }
\end{theorem}

The boundary sources 
$f_n{{(\alpha,\beta}},k)$ in Theorem \ref{Thm 1}, that produce an artificial point source, are
obtained using an iterative sequence of measurements. In this iteration, we first
measure for $n=1$ the boundary value, $\Lambda f_1$, of the wave  that is produced by 
a certain boundary source $f_1$.
In  each iteration step, we use the boundary source $f_n$ and its response $\Lambda f_n$
to compute the boundary source $f_{n+1}
$ for the next   iteration step.  
The iteration algorithm in this paper was inspired by time reversal methods, see \cite{ Bal4, Bal2,Papa1,Papa3,Fink2006,FinkMain,Jon-DeHoop,MNW,Papa2,PTF}.  
We note that when the  traditional time-reversal algorithms are used in imaging, one typically needs to assume that
the medium contains some point-like scatterers.

 Generally, when the coefficients of the operator $\A$ are unknown, 
one can not specify the Euclidean coordinates of the point $\hat x$ to which the waves focus,
but only the Riemannian boundary normal coordinates $(\hat z,\hat t)$
(called also the ray coordinates in optics or the migration coordinates in Earth sciences) of $\hat x$ can be specified.
However, in the case when $M\subset \R^m$ and the operator $\mathcal A$ is of the form $\mathcal A=-c(x)^2\Delta$, we show in Corollary \ref{cor: focusing coordinates}
that the Euclidean coordinates of the point $\hat x$ can be computed using the
Neumann-to-Dirichlet map $\Lambda$.

%
%
%
%
%
%
%

The problem studied in the paper is motivated by recent advances in the applications  of optimal control methods to 
\emph{lithotripsy} and  \emph{hyperthermia}.
In lithotripsy, one breaks down a kidney or bladder   stone using a focusing ultrasonic wave.  Likewise, in \emph{hyperthermia} in medical treatments, cancer tissue is destroyed by {ultrasound  induced heating} that produces an excessive heat dose generated by a focusing wave \cite{MaHuttKaipio}.  Often, to apply these methods one needs to use other physical imaging modalities, for example X-rays tomography of MRI to estimate the material parameters in $M$.
However, for   the wave equation there are various methods to estimate material  parameters using boundary measurements of waves. These methods are, however, quite unstable   \cite{AKKLT,KatsudaKurylevLassas}. Therefore they might not be suitable  for hyperthermia, where safety is crucial. An important question is therefore how to focus waves in unknown media.


In the paper we 
advance further  the techniques  developed in \cite{BKLS} and \cite{DKL}.
In \cite{DKL}, a construction of focusing waves was considered in the analogous setting to this paper, but using 
the function space $L^2(M)\times L^2(M)$  instead of the natural energy space
$H^1_0(M)\times L^2(M)$ in \eqref{eq: A main limit}. The use of the function space associated to energy makes it possible to concentrate the energy of the wave near a single point.
For instance in the above ultrasound  induced heating problem,
the use correct energy norm is crucial as otherwise the energy of the wave may not
be concentrating at all. 

The other novelties of the paper are that in the case
of isotropic medium, that is, with the operator  $\mathcal A=-c(x)^2\Delta$ we can 
focus the wave near a point $\hat x$ which Euclidean coordinates can be computed (a posteriori).
We apply this to an inverse problem, that is, for determining the wave speed
in the unknown medium.

The methodology in this paper arises from boundary control methods used to study inverse problems in hyperbolic equations \cite{AKKLT,B1,BK,Kian,MR3924852,KKL,KKLM,LO} and on focusing of waves for non-linear equations
\cite{deHoop_et_al,Feiz,Hintz,Krupchyk,KLUinv,Lassas-ICM,WZ}. 
Similar problems have been studied using 
geometrical optics \cite{Rakesh1988,Rakesh1990,Rakesh2011,Stefanov2005} and the methods of scattering theory \cite{Caday},
see also the reviews of these methods in \cite{UhlmannICM,Uhlmann2004}.

In particular, Theorem \ref{Thm 1} provides for linear equations an analogous 
construction of the artificial point sources that is developed in  \cite{KLUinv}  for non-linear
hyperbolic problems with a time-dependent metric. We note that this technique is used as a surprising example on how
the inverse problems for non-linear equations are sometimes easier than for the corresponding problems for the 
linear equations. Thus  Theorem \ref{Thm 1}  shows that some tools that are developed for inverse problems for non-linear equations
can be generalized for linear equations.

The outline of this work is as follows.  In Section \ref{Terminology} we introduce notation, boundary control operators and review some relevant results from control theory. In Section \ref{sectionMIN} we state and describe the minimization problem for the boundary sources. In Section \ref{sectionFOCUS}, we discuss focusing of the waves and prove Theorem \ref{Thm 1}.  In Section \ref{sectionITERATION} we introduce the modified iteration time-reversal scheme to generate boundary sources using an iteration of simple operators and boundary measurements.  In Section  \ref{sec_computations} we present the results of the numerical experiment. In Section \ref{sectionDDF} we apply the results for  inverse problems.  Some of the proofs can be found in Appendices.

\section{Definitions}
\label{Terminology} \subsection{Manifold $M$} We assume that $M$ is closed $C^\infty$-smooth bounded set in $\R^m$ $(m\ge 1)$ with non-empty smooth boundary $\p M$ or an $m$-dimensional $C^\infty$-smooth compact manifold with boundary.  Furthermore, we  assume that $M$ is equipped with  a $C^\infty$-smooth Riemannian metric
$\texttt{g}=\sum_{j,k=1}^m \texttt{g}_{jk}(x) \,dx^j \otimes dx^k$. Elements of the inverse matrix of $\texttt{g}_{ij}$ are denoted by $\texttt{g}^{ij}$.
Let $\mathrm{dV}_\texttt{g}$ be the smooth measure
$
  \mathrm{dV}_{\texttt{g}} = \vert \texttt{g}(x)\vert^{1/2}dx^1
  \cdots  dx^m,
$
where $|\texttt{g}|=|\texttt{g}(x)|=\det([\texttt{g}_{jk}])$. Then  the inner product in $L^2(M)$ is defined by the inner product
\begin{align*}
  \bra u,v\cet_{L^2(M)} = \int_{M} u(x)v(x)\, \mathrm{dV}_\mu(x),
\end{align*}
where $\mathrm{dV}_\mu(x) = {\mu}(x) \mathrm{dV}_\texttt{g}(x)$ and
$\mu \in C^\infty(M)$ is a  strictly positive function on $M$.

We assume that $\A,$ introduced in \eqref{eq: Wave}, represents a general formally selfadjoint elliptic second order differential operator such that its  potential term vanishes (see \cite{KKL} for the details).  In the local coordinates, $\A$ can be represented in the form 
\begin{equation}\label{A}
\A v=- \sum_{j,k=1}^m \frac{1}{\mu(x) |\texttt{g}(x)|^{1/2}} \frac {\p}{\p
x^j}\left( \mu(x) |\texttt{g}(x)|^{1/2}\texttt{g}^{jk}(x)\frac {\p v}{\p x^k} \right).
\end{equation}
For example, if $\mu=1$  then $\A$ reduces to the  Riemannian Laplace operator.  

On the boundary $\p \M$, operator $\p_\nu$ is defined by
\ba
\p_\nu v=\sum_{j=1}^m \mu(x)\nu^j\frac {\p}{\p x^j}v(x)
\ea
where $\nu(x)=(\nu^1,\nu^2,\dots,\nu^m)$ is the interior unit normal
vector of the boundary satisfying $\sum_{j,k=1}^m \texttt{g}_{jk}(x)\nu^j\nu^k=1$.  {\newtext To
  integrate functions on $\p M$ we use the {\newtexttt measure}
  $\mathrm{dS}=\mu \mathrm{dS}_g$ on $\p M$ induced by $\mathrm{dV}_\texttt{g}$.  }
If $\Omega\subset\p M\times \mathbb{R}_+$, we denote
$
  L^2(\Omega)=\{f\in L^2(\p M\times\mathbb{R}_+) : \operatorname{supp}(f)\subset
\overline \Omega\},
$
identifying functions and their zero continuations.

%


\subsection{Travel time metric} \label{sec22}
Let $\texttt{d}(x,y)$ be the \emph{geodesic distance} corresponding to $\texttt{g}$.
The metric $\texttt{d}$ is also called the \emph{travel time metric} because
it describes how solutions of \newtextttt{the wave equation}
propagate. When $\Gamma\subset \p M$ is open, and $f\in L^2(\Gamma\times\mathbb{R}_+)$, then at time $t>0$, by finite velocity of wave propagation, solution $u^f(\,\cdotp,t)$ is supported in the \emph{domain of influence}  (see \cite{HormanderIV})
\begin{eqnarray*}
  M(\Gamma,t)=\{x\in M :\texttt{d}(x,\Gamma)\leq t\}.
\label{domain_of_influence}\end{eqnarray*}
{\newtext
The \emph{diameter} of $M$ is defined as
$
  \operatorname{diam}(M) = \max{\{\texttt{d}(x,y) : x,y\in M\}}.\label{diameter_definition}
$

\newt{
Let  $T_x M$ be the tangent space of $M,$ $x\in M$ and   $\xi\in T_x M$, $\|\xi\|_g=1$. We denote by $\gamma_{x,\xi}(s)$ the geodesic in $ M$, which is parameterized with its arclength and satisfies $\gamma_{x,\xi}(0)=x$  and $\dot\gamma_{x,\xi}(0)=\xi$. 
Suppose $z\in{\p M}$ and $\nu=\nu(z)$ is the interior unit normal
vector at $z\in \p M$. Then a geodesic $\gamma_{z,\nu}$ is called a 
\emph{normal geodesic}, and there is a critical value $\tau_{\p M}(z)>0$, such
that for $t<\tau_{\p M}(z)$ geodesic $\gamma_{z,\nu}([0,t])$ is
the unique shortest curve in $M$ that
connects $\gamma_{z,\nu}(t)$ to ${\p M}$,  and for
$t>\tau_{\p M}(z)$ this is no longer true. More precisely, we define the \emph{critical value} 
\beq\label{critical distance}
  \tau_{\p M}(z)=\sup \{ s>0 : \texttt{d}(\gamma_{z,\nu}(s),{\p M})=s\}.
\eeq
}

%
}

\subsection{Controllability for wave equation}
\label{Sec: Controllability}
Let us denote $u^f (T)=u^f (\,\cdotp,T)$.
The seminal Tataru's unique
continuation result \cite{Ta1} implies 
the following approximate controllability result:
%
%
%

\begin{proposition}[Tataru's approximate global controllability]\label{global controllability}
\newtextttt{
\label{lemma_Control 1} Let
 $T>2\, \operatorname{diam}(M)$. Then the linear subspace
$
  \{(u^f (T), u^f_t(T)) : \,f\in C_0^\infty ({\p M} \times (0,T))\}
$
is dense in $H^1_0(M)\times L^2(M)$. }
\end{proposition}

The proof of Proposition \ref{global controllability} is given in \cite[Thm. 4.28]{KKL}.

 Tataru's unique
continuation result implies also the
following local controllability result. The \emph{indicator function} of a set $S$ is denoted by $\mathds{1}_S$.

\begin{proposition}[Tataru's approximative local controllability]
\newtextttt{
\label{lemma_Control_2} Let $T>0$,  let $\Gamma_1, \ldots, \Gamma_J \subset \p M$ be non-empty open sets, and let $0<s_k\le T$ for $k=1,\ldots, J$. Suppose 
\begin{eqnarray}
\label{B-set}
  \mathcal B=\bigcup^{J}_{j=1} \Gamma_j\times (T-s_j,T),\quad\mathcal N=\bigcup^{J}_{j=1}M(\Gamma_j,s_j),
\end{eqnarray}
{\newtexttt
and $P$ is multiplication by the indicator function $\mathds{1}_{\mathcal B}$,
}\label{dense1}
\begin{eqnarray}
\label{B-projection} 
 P\colon L^2({\p M} \times (0,2T)) \rightarrow L^2({\p M} \times(0,2T)), \quad
(P f)(x,t)= \mathds{1}_{\mathcal B}(x,t)\,f(x,t).\hspace{-1cm}
\end{eqnarray}
Then the linear subspace
$
  \left\{u^{Ph}(T) : h \in L^2 ({\p M}\times (0,2T) ) \right\}
$
is dense in $L^2(\mathcal N).$} 
\end{proposition}

Proposition \ref{lemma_Control 1} follows directly from
\cite[Thm. 3.10] {KKL}.

\subsection{Auxiliary operators}\label{sec_boundary_ops}  In this section we introduce several operators to manipulate boundary sources.  

Let $h\in L^2\left(\p M\times (0,2T)\right)$ be the Neumann boundary value (a source function). Then 
by \cite[Thm. A]{Lasiecka2}), the initial-boundary value problem (\ref{eq: Wave}) has a unique solution $u^h$ and we define a map 
\begin{align}
\label{map_for_h}
& U:L^2\left(\p M\times (0,2T)\right)\to
C\left([0,2T];H^{3/5-\varepsilon}(M)\right),\quad
 U:h \mapsto u^h,
\end{align}
where $\varepsilon>0$.
We define also the space
\ba
H_0^1 \left((0,T);L^2(\p M)\right)=\{
f\,|\,[0,T]\times \p M\to \R&:&\ 
f,\p_t f\in L^2([0,T]\times \p M),\\
& &
f(x,t)|_{t=0}=0,\  f(x,t)|_{t=T}=0
\}.
\ea

Let $a\in H_0^1\left((0,T);L^2(\p M)\right)$ be another Neumann boundary value, then solution of the initial-boundary value problem (\ref{eq: Wave}) defines 
a bounded map
\begin{align}
\label{map_for_a}
U:H_0^1\left((0,T);L^2(\p M)\right)\to
C\left([0,2T];H^{3/2}(M)\right),\quad 
U:a \mapsto u^a,
\end{align}
see \cite[Thm. 3.1(iii)]{Lasiecka2}.
\subsubsection{Sobolev spaces on the boundary}  Let us introduce Sobolev spaces \ba
V=L^2\left(\p M\times(0,2T)\right),\quad  Y=H^1_0\left((0,T);L^2(\p M)\right),\quad  Z=H^1_0\left((0,2T);L^2(\p M)\right),\ea 
%
while the inner product in $Z$ is given by 
$
\bra a_1, a_2\cet_{Z}=\bra a_1, a_2\cet_V+\bra \p_t a_1, \p_t a_2\cet_V .
$

\subsubsection{Neumann-to-Dirichlet map} 
For $h\in V$, the boundedness of the map  (\ref{map_for_h}) implies that the trace of solution
satisfies $u^h |_{\p M \times(0,2T)}\in C\left([0,2T];H^{3/5-1/2-\varepsilon}(\p M)\right)$, where $\varepsilon>0.$ Hence the Neumann-to-Dirichlet map
\begin{align}
\label{Lam-operator}
 & \Lambda\colon V\to V, \quad\Lambda h =  u^h |_{\p M \times(0,2T)}
\end{align}
is a bounded linear operator, where $u^h$ is the solution of \eqref{eq: Wave}.

\subsubsection{Time-reversal map and time filter map}
Let  
\begin{eqnarray*}
 R:V\to V,\quad Rf(x,t)=f(x,2T-t), 
 \end{eqnarray*} be the \emph{time reversal map} and 
\begin{eqnarray}
\label{J-operator}
J:V\to V,\quad Jf(x,t)= \frac 12\int_{[0,2T]} \mathds{1}_{\mathcal L}(s,t)f(x,s)ds,
\end{eqnarray}
be the \emph{time filter map}, where 
\begin{eqnarray}\label{shapeL}
  \mathcal L&=&\{ (s,t)\in \R_+\times\R_+:\ t+s\leq 2T,\ s>t \}.
\end{eqnarray}
The adjoint  $\Lambda^*\colon V\to V$, of the Neumann to Dirichlet map 
$\Lambda\colon V\to V$, is 
$ \Lambda^*=R\Lambda R,$
see \cite[eq. 21]{BKLS}.

\subsection{Blagovestchenskii identities}\label{secBlag} The inner product of solutions of (\ref{eq: Wave}) at time $T$, i.e. waves $u^f(\,\cdotp,T)$ and  $u^h(\,\cdotp,T),$ generated by two boundary sources $f, h$ can be calculated from boundary measurements on $\p M$ using the identity below.    For $f,h\in V$ the \emph{first Blagovestchenskii identity} states that
\begin{equation}
\label{blago1}
\int_{M} u^f(T)u^h(T)\,\hbox{dV}_\mu =\int_{\p M \times [0,2T]}
(Kf)(x,t) h(x,t)\, \hbox{dS}_g(x) dt, 
\end{equation}
where $ \hbox{dS}_g$ is the Riemannian volume on $\p M$, and $K$  is defined
in terms of the Neumann-to-Dirichlet map $\Lambda$ and simple operators on boundary as \begin{eqnarray}
\label{K-operator}
K\colon V\to V,\quad K=R\Lambda R J-J\Lambda,
\end{eqnarray}
see \cite[eq. 23]{BKLS}. 
The \emph{second Blagovestchenskii identity} is \begin{align}
\label{Tformula2}
 \bra u^{h}(T),1 \cet_{L^2(M)}=-\bra h,\Phi _T \cet_V,
\end{align}
where 
 $\Phi _T:\left({\p M} \times (0,2T)\right)\to \R$ is the function
\beq
\label{ihk}
\Phi _T(x,t)=(T-t)_+=\begin{cases}
T-t,\quad\quad t\le T\\
0,\quad\quad\quad t>T.
\end{cases}\eeq
The proofs for formulas (\ref{blago1}) and  (\ref{Tformula2}) can be found \cite[Lemma 1]{BKLS} and \cite{DKL}, see also the 
Appendices \ref{appA1} and \ref{appA2}.




\subsubsection{Projection Operators}  We use frequently the projection operator $P=P_{\mathcal B}$ introduced in (\ref{B-projection}). We define also an orthogonal projection in $Z$ (a {\it support shrinking projector})
\begin{eqnarray}
\label{P0-projection}
{N_Y}:  Z\rightarrow  Z,\qquad \operatorname{Ran}({N_Y}) =  Y\subset  Z.
\end{eqnarray}
Note that ${N_Y}$ can be written also
\begin{equation}\nonumber\label{eq:hypsine}
{N_Y} f \,=\, \underset{u\in H_0^1 \left((0,T);L^2(\Gamma)\right)}{\argmin} \|f-u\|^2_{H_0^1 \left((0,2T);L^2(\Gamma)\right)},
\end{equation}
and it is given by 
$
{N_Y} f=f(x,t)-\frac{\sinh(t)}{\sinh(T)}f(x,T).
$  
Additionally, we introduce a projection
  \begin{eqnarray}
\label{T-projection}
 \hat P\colon V\to V,\quad
  (\hat P f)(x,t)=\mathds{1}_{{\p M}\times (0,T)}(x,t)\,f(x,t).
\end{eqnarray}
\subsubsection{Green's operator on the boundary}  
Let
\begin{align}
\label{Q_oper}
Q: V \to Z,\qquad
Qf(x,t)=\int_0^{2T} g(t,s)f(x,s)ds,
\end{align} 
where $g\colon (0,2T)^2 \to \R$,
$$
g(t,s)=\frac{1}{2(e^{4T}-1)}
  \begin{cases}
    (e^{t}-e^{-t})(e^{4T}e^{-s}-e^{s}),\quad t<s,\\
    (e^{s}-e^{-s})(e^{4T}e^{-t}-e^t),\quad t>s,
  \end{cases}\, \,
$$
is the Green's function for the problem
$$
  \begin{cases}
    (1-\p_t^2)g(t,s)=\delta(t-s),\quad t\in (0,2T)\\
    g|_{t=0}=0,\quad  g|_{t=2T}=0, \quad 
  \end{cases}
$$
where $s\in (0,2T)$. Note that $Q\colon V\to Z$ is bounded.

\section{Minimisation Problems}
\label{sectionMIN}
 Let
$P$  be the projector   given in \eqref{B-projection} associated to the sets $\mathcal B$ and $\mathcal N$ given in
\eqref{B-set}. We will consider two minimization problems.
The first one is considered to find $h\in V$ such that
$u^{Ph}(T)$ is close to the indicator function $\mathds{1}_{\mathcal N}$ in $L^2(M)$.
The second minimization problem is considered to find $a\in Y$ such that the time derivative
$u^{a}_t(T)$ is close to $u^{Ph}(T)$ and therefore close to $\mathds{1}_{\mathcal N}$ in $L^2(M)$
and that value of the wave $u^{a}(T)$ is close to zero in $H^1_0(M)$.

%
%
%
%
To consider the first minimization problem,
 we define for $\alpha\in(0,1)$  the quadratic form $h\mapsto \mathcal F_1(h,\a)$, 
\beq
\label{possupahka1}
\mathcal F_1(h;\a) = \|\mathds{1}_{\mathcal N}-u^{Ph}(T)\|^2_{L^2(M)} +\, \alpha \|h\|^2_V,
\quad h\in V
\eeq
Then, we define $h_\alpha\in V$ be the minimizer 
 \beq\label{minimizer a}
h_\alpha=\underset{h\in V}{\argmin}\,\mathcal F_1(h;\alpha).
\eeq

To consider the second minimization problem, for
 $\beta\in(0,1)$, $h\in V$  we define 
\begin{align}
\label{possupahka2}
\mathcal F_2(a; \beta,h) &=&\|u_t^{a}(T)-u^{Ph} (T) \|^2_{L^2(M)}  + \|u^{a}(T)\|^2_{H^1(M)} 
+\, \beta\|a\|^2_{Y},\quad a\in Y.
\end{align}
We minimize this functional with respect to $a$ when $h=h_\alpha$,
and define
 \beq\label{minimizer ab}
 a{{(\alpha,\beta}})=\underset{a\in Y}{\argmin}\,\mathcal F_2(a;\beta, h_\a).
 \eeq




Using the Blagovestchenskii identities (\ref{blago1}) and  (\ref{Tformula2}) we rewrite $\mathcal F_1(h,\a)$ and $\mathcal F_2(a,\beta|h)$ in terms 
 that, up to a constant term, can be computed on the boundary, 
 \begin{equation}\label{eq: minimize 41}
\mathcal F_1(h;\a)=\bra \mathds{1}_{\mathcal N},\mathds{1}_{\mathcal N}\cet_{L^2(M)}+2\bra Ph,\Phi _T\cet_V + \bra Ph,KPh\cet_V +\alpha \bra h,h \cet_V.
\end{equation}
\begin{eqnarray}\label{eq: minimize 42}
\mathcal F_2(a;\beta,h)&=& \bra Ph,KPh\cet_V-2\bra Ph,K\p_t a\cet_V+\bra\p_t a,K\p_ta\cet_V \\\nonumber
  &&-\quad 2\bra a,\hat P\p_t\Lambda a \cet_V- \bra \p_t a,K\p_t a\cet_V+\bra a,Ka\cet_V.
   \end{eqnarray}
Next we consider how $h_\alpha$ and $a{{(\alpha,\beta}})$ can be found using
the map $\Lambda$.

%
 %

\begin{theorem}\label{min_theorem_h} For $\a\in(0,1),$ the solution of
the equation
\begin{equation}\label{eq1}
\left(
PKP+\a
\right)h=-P\Phi_T
\end{equation}
is the unique minimizer $h_\alpha\in V$ of $\mathcal F_1(h;\a)$ in the space $h\in V$,  see (\ref{possupahka1}).  Furthermore, map $PKP\colon V\to V$ is non-negative, bounded, and selfadjoint. Moreover $\|h_\a\|^2_V\le \frac{1}{\alpha}(1+T)^2.$
\end{theorem}
\begin{proof}
First, we recall that operators  $K$  (\ref{K-operator}) and $P$ (\ref{B-projection}) are bounded operators $V\to V$ and hence $PKP:V\to V$ is bounded. Since $PKP$ is non-negative and selfadjoint, $\F_1$ is strictly convex and the minizer is unique.
 Using \eqref{eq: minimize 41} we see that
%
the Fr\'echet derivative of $h\mapsto \mathcal F_1(h;\a)$ at $h\in V$ in the  direction $\eta\in V$ is given by 
$$
D\mathcal F_1(\,\cdotp;\a)|_h\eta=\bra \eta, (PKP+\a)h +P\Phi\cet_V .$$ For a fixed $\a$, the Fr\'echet derivative is zero when the boundary source function $h_\a$ is a solution of (\ref{eq1}),  and $h_\a$ is the minimizer for the functional (\ref{possupahka1}). Note that $PKP+\alpha I\ge \alpha I$ and $\|(PKP+\alpha I)^{-1}\|_V\le \frac{1}{\alpha}.$
\end{proof}

\begin{theorem}\label{min_theorem_a} Let $h_\alpha\in V$ be the solution of the equation (\ref{eq1}). For $\beta\in(0,1),$ the unique minimizer  $a=a{{(\alpha,\beta}})\in Y$
of the functional $\mathcal F_2(a;\beta, h_\a)$, see (\ref{possupahka2}),
is the solution of the equation
\begin{equation}\label{eq2}
(L+\beta)a=-{N_Y}Q\p_t KPh_\a
\end{equation}
 where 
\begin{align}
\label{L-oper}
L\colon Y\to Y, \quad L={N_Y}Q\,\left(R\Lambda R \p_t-\hat P\p_t\Lambda +K\right).
\end{align}
Furthermore, $L\colon Y\to Y$ is non-negative, bounded, and selfadjoint.
\end{theorem}  


To prove the above result, we first consider the energy function and prove auxiliary Lemmas \ref{energy_lemma} and \ref{inner_lemma}, and then we prove Theorem \ref{min_theorem_a}.

We observe that  as  for $a\in Y$ we have $a=\hat P a$,
we can write the operator $L$ in \eqref{L-oper} in a more symmetric form
\beq\label{formula with hat P}
La={N_Y}Q\,\left(R\Lambda R \p_t\hat P-\hat P\p_t\Lambda +K\right)a,\quad a\in Y.
\eeq

\begin{definition}
Let us define the energy function in the following way
\begin{align}\label{eat}
E(a,t)=\|u_t^{a} (t) \|^2_{L^2(M)}+\|\nabla_{\texttt{g}} u^{a} (t) \|^2_{L^2(M)},\quad t\in (0,2T).
\end{align}
\end{definition}  Hence we can replace the second term in (\ref{possupahka2}) using the identity
\begin{align}
\label{Tformula12}
 \|u^{a}(T)\|^2_{H^1(M)}=E(a,T)-\|u_t^{a} (T) \|^2_{L^2(M)}+\|u^{a} (T) \|^2_{L^2(M)}.
\end{align} The benefits of doing this can be seen from the following Lemma.
\begin{lemma}
\label{energy_lemma} For $a\in V$ and energy function defined in (\ref{eat}) satisfies
\begin{align} 
\label{Tformula11}
 E(a,T)=-2\bra a,\hat P\p_t\Lambda a \cet_V.
\end{align}
\end{lemma} 

\begin {proof} Using (\ref{eat}) we get
\ba
E(t)
&=&\int_{M}[\p_tu^{a}(x,t)\p_tu^{a}(x,t)+\underset{j,k=1}{\overset{n}{\sum}}\texttt{g}^{jk}(x)\p_{x_j}u(x,t)\p_{x_k}u(x,t)]\,\mu(x)\mathrm{dV}_\texttt{g}(x).
\ea
Differentiation respect the time and integration by parts gives us
\ba
\p_tE(t)
\hspace{-2mm}&=&\hspace{-2mm}2\int_{M}\left[\p_t^2u^{a}(x,t)\p_tu^{a}(x,t)+\underset{j,k=1}{\overset{n}{\sum}}\texttt{g}^{jk}\p_{x_j}u(x,t)\p_{x_k}\p_tu(x,t)\right]\,\mu(x)(\mathrm{det}\, \texttt{g})^{\frac{1}{2}}\mathrm{d}x\\
\hspace{-2mm}&=&\hspace{-2mm}-2\int_{\p M}\left[a(t)\p_t\Lambda a(t)\right]\,\mathrm{dS}_g(x).
\ea
At time $t=0$ we have the initial values $\p_t u(x,0)=0$ and $u(x,0)=0$, and thus  $E(0)=0$. Thus
\ba
E(T)
=-2\int_{0}^{T}\int_{\p M}\left[a(t)\p_t\Lambda a(t)\right]\,\mathrm{dS}_g(x)\,\mathrm{d}t
= -2\bra a,\hat P\p_t\Lambda a \cet_V.
\ea
\vspace{-4mm}

\end{proof}


\begin{lemma}
\label{inner_lemma} Let Q be given in (\ref{Q_oper}) and ${N_Y}$ be the projector in (\ref{P0-projection}). For $a\in Y$ and $f\in V$ we have
\begin{align} 
\label{Tformula41}
\bra {N_Y}Qf, a\cet_{Y}=\bra f, a\cet_{V}.
\end{align}
\end{lemma}  
\begin{proof} Let $f\in V$ and $a\in Y$. 
 Definition (\ref{Q_oper}) also implies that $Qf\in Z$ and 
\begin{align}
\label{isosika}
\bra {N_Y}Qf, a\cet_{Y}=\bra Qf, {N_Y}a\cet_{Y}=\bra Qf, a\cet_{Y}=
\bra Qf
,a\cet_{Z}=\bra (1-\p_t^2)Qf
,a\cet_{V}=\bra f
,a\cet_{V}. 
\end{align}
\end{proof}
\begin{proof}[Proof of Theorem \ref{min_theorem_a}]
Let us first show that 
 $\pat KPh_\a\in V$.
%
To this end, observe that 
 $J$ increases smoothness  the time variable by one, that is,  $J: V\to H^1 \left((0,2T) ,L^2({\p M})\right)$.
 
 Moreover, by the definition of the set $\mathcal L$ in \pef{shapeL},
 we see that $RJh_\a|_{t=0}=0$ and $Jf(x,2T)=0$.
 First, this shows that  $J \HL Ph_\a\in H^1 \left((0,2T), L^2({\p M})\right)$.
 Second, 
 as by  \cite[Thm. 3.1(iii)]{Lasiecka2} and the trace theorem
 we have
\begin{eqnarray*}
\Lambda:\{a \in H^1\left((0,2T); L^2({\p M})\right):a(x,0)=0
\}\to C^1 \left([0,2T]; H^{\frac35-\frac12-\varepsilon} ({\p M})\right),
\end{eqnarray*}
we see that $R\HL RJP h_\a\in C^1 ([0,2T]; H^{\frac35-\frac12-\varepsilon}({\p M}))$.
These show that
$\pat KPh_\a\in V$. Hence, we have
%
%
%
 ${N_Y}Q\p_tKP h_\a\in Y$. To continue the proof, we need the following lemma.

\begin{lemma}
The operator $L:Y\to Y, L={N_Y}Q\,\left(R\Lambda R \p_t-\hat P\p_t\Lambda +K\right)$ is  bounded.
\end{lemma}
\begin{proof}
Note that $K:Y\mapsto V$ is bounded.
 Also $\p_t:Y\to V$ is  bounded. Let $a\in Y$. The boundedness of the operator $R\Lambda R:V\to V$ implies $R\Lambda R \p_t  a\in V$. Let $a\in H_0^1\left((0,T);L^2({\p M})\right)$. 
  Due to
\cite[Thm. 3.1 (iii)]{Lasiecka2} and the trace theorem, we have 
 \begin{align}
\label{solution-map2}
\Lambda a\in C^1\left([0,2T];H^{\frac35-\frac12-\ve}({\p M})\right),\quad\ve>0,
\end{align} 
and thus $\hat P\p_t\Lambda:Y\to V$ is bounded. The map $(R\Lambda R \p_t-\hat P\p_t\Lambda +K):Y\to V$ is bounded. Using definitions of $Q$  in (\ref{Q_oper}) and ${N_Y}$ in (\ref{P0-projection}), we see that ${N_Y}Q(R\Lambda R \p_t-\hat P\p_t\Lambda +K):Y\to Y$ is bounded. 
\end{proof}
\begin{lemma}
The operator $L:Y\to Y$ is selfadjoint and non-negative.
\end{lemma}
\begin{proof} Below we use formula \eqref{formula with hat P} several times.
For $f_1,f_2\in Y,$ due to Lemma \ref{inner_lemma} we have
\ba
  \bra {N_Y}Q(R\Lambda R \p_t\hat P-\hat P\p_t\Lambda +K)f_1,f_2\cet_{Y}= \bra (R\Lambda R \p_t\hat P-\hat P\p_t\Lambda +K)f_1,f_2)\cet_V.
\ea
Since operator $K:V\to V$ is selfadjoint and since $\Lambda^*=R\Lambda R,$ we have 
\ba
 \bra (R\Lambda R \p_t\hat P-\hat P\p_t\Lambda +K)f_1,f_2\cet_V = \bra f_1,(R\Lambda R \p_t\hat P-\hat P\p_t\Lambda +K)f_2\cet_V.
\ea
By  Lemma \ref{inner_lemma}, we get
\ba
\bra f_1,(R\Lambda R \p_t\hat P-\hat P\p_t\Lambda +K)f_2\cet_V = \bra f_1,{N_Y}Q(R\Lambda R \p_t\hat P-\hat P\p_t\Lambda +K)f_2\cet_{Y}.
\ea
Thus for $f_1,f_2\in Y$, we conclude that $ \bra Lf_1,f_2\cet_{Y} =\bra f_1,Lf_2\cet_{Y}. $
This proves that  $L:Y\to Y$ is selfadjoint.

Next we show that the operator $L:Y\to Y$ is non-negative. 
Recall that  $\Lambda^*=R\Lambda R:V\to V.$ Thus for $f\in Y$ we have
\ba
\bra f,Lf\cet_{Y}=\bra f,(\Lambda^* \p_t\hat P-\hat P\p_t\Lambda +K)f\cet_V = -2\bra f,\hat P\p_t\Lambda f \cet_V+\bra f,Kf\cet_V.
\ea
By definition for energy (\ref{eat}) and Blagovestchenskii identity (\ref{blago1}) we have
\ba
-2\bra f,\hat P\p_t\Lambda f \cet_V+\bra f,Kf\cet_V= E(f,T)+\|u^{f} (T) \|^2_{L^2(M)}\ge 0.
\ea
Therefore, for $f\in Y$, we  have showed that $ \bra f,Lf\cet_{Y}\ge 0. $ 
\end{proof}

Next we rewrite $ \mathcal F_2 (a;\beta , h_\a)$ in (\ref{eq: minimize 42})  by using
equations (\ref{Tformula12}), (\ref{Tformula11}), \emph{Blagovestchenskii identitities} (\ref{blago1}), (\ref{Tformula2}), and $\p_t u^a(T)= u^{\p_ta}(T)$.
These yield that
\ba
 \mathcal F_2 (a;\beta , h_\a)&=& \bra Ph_\a,KPh_\a\cet_V-2\bra Ph_\a,K\p_t a\cet_V\\ 
\nonumber
  &&-2\bra a,\hat P\p_t\Lambda a \cet_V+\bra a,Ka\cet_V+\beta\bra a,a \cet_{Y}. \label{eq: minimize 15}
\ea
As operators $K: V \to V$ and $\hat P: V \to V$ are selfadjoint, and $\Lambda^*=R\Lambda R$,
\ba
  \mathcal F_2 (a;\beta , h_\a)&=&\bra Ph_\a,KPh_\a\cet_V+2\bra \p_t KPh_\a,a\cet_V-\bra a,\hat P\p_t\Lambda a \cet_V\\ 
\nonumber
  &&+\bra a,R\Lambda R \p_t\hat P a \cet_V+\bra a,Ka\cet_V+\beta\bra a,a \cet_{Y}. \label{eq: minimize 16}
\ea
Further, Lemma \ref{inner_lemma} implies that $ \mathcal F_2 (a;\beta , h_\a)$ can be written in the form
\ba
  \mathcal F_2 (a;\beta , h_\a)&=&\bra Ph_\a,KPh_\a\cet_V+2\bra {N_Y}Q\p_t KPh_\a,a\cet_{Y}-\bra a,{N_Y}Q\hat P\p_t\Lambda a \cet_{Y}\\ \nonumber
  &&+\bra a,{N_Y}QR\Lambda R \p_t\hat P a \cet_{Y}+\bra a,{N_Y}QKa\cet_{Y}
  +\beta\bra a,a \cet_{Y}, \label{eq: minimize 17}
\ea
whereas the latter can be written as
\ba
  \mathcal F_2 (a;\beta , h_\a)=\bra Ph_\a,KPh_\a\cet_V+\bra (L+\beta)a+2{N_Y}Q\p_t KPh,a\cet_{Y}.\label{eq: minimize 18} 
\ea
The operator $L\colon Y\to Y$ is non-negative, bounded, and selfadjoint.  
Thus the functional $a\mapsto  \mathcal F_2 (a;\beta, h_\a)$ is strictly convex.
Hence the unique minimum of   $\mathcal F_2 (a;\beta , h_\a)$ is at the zero of the Fr\'echet derivative
of $a\mapsto   \mathcal F_2 (a;\beta , h_\a)$ at $a$ given by
\ba
 D  \mathcal F_2 (\cdotp;\beta , h_\a)|_a \xi =
\bra  {N_Y}Q\p_t KPh + (\beta + L)a, \xi  \cet_{Y},
\quad \xi\in Y.
\ea
The Fr\'echet derivative is zero when the boundary source  $a$ is the solution of the equation  (\ref{eq2}). 
Thus $$a(\beta, h_\a)= -(\beta + L)^{-1}{N_Y}Q\p_t KPh$$ is the minimizer for the functional (\ref{possupahka2}). This completes the proof of Theorem \ref{min_theorem_a}.
\end{proof}


\begin{lemma}
\label{finalLemma} 
Let $T>2\operatorname{diam}(\M)$ and let $h_\alpha\in V$ and
$a(\beta, h_\alpha) \in Y$ be the minimizers of 
(\ref{possupahka1}) and  
 (\ref{possupahka2}), respectively. Then 
\begin{eqnarray*}
  \lim_{\alpha\to 0} \lim_{\beta\to 0}
  \vc{  u^{a(\beta, h_\alpha)}(T)}{u^{a(\beta, h_\alpha)}_t(T)} = \vc{0}{\mathds{1}_{\mathcal N} },
\end{eqnarray*}
where limits are in $H_0^1(\M)\times L^2(\M),$ and $\mathcal N$ is defined in (\ref{B-set}). 
\end{lemma}
\begin{proof}
Proposition \ref{dense1} implies that for  any $\varepsilon >0$ there is $h(\varepsilon)in V$ such that
\ba
\|\mathds{1}_{\mathcal N} -u^{Ph(\varepsilon)} (T)\|_{L^2(M)}^2 < \varepsilon.
\ea
On the other hand, for every $\alpha \in (0,1)$, the minimizer $h_\alpha$ satisfies
\ba
\|\mathds{1}_{\mathcal N}-u^{Ph_\alpha}(T)\|_{L^2(M)}^2 +\alpha\|h_\alpha\|^2_V\le
\|\mathds{1}_{\mathcal N} -u^{Ph(\varepsilon)} (T)\|_{L^2(M)}^2+\alpha\|h(\varepsilon)\|^2_V.
\ea
If $\alpha\le\alpha_0(\varepsilon)=\frac{\varepsilon}{1+\|h(\varepsilon)\|^2_V}$, we have
$
\|\mathds{1}_{\mathcal N}-u^{Ph_{\alpha}}(T)\|_{L^2(M)}^2 \le 2\varepsilon,
$
and hence
\begin{align}
\label{sikaniska}
u^{Ph_\alpha} (T)\to \mathds{1}_{\mathcal N}\quad\mbox{in $L^2(M)$ as}\quad\alpha \to 0.
\end{align}

By Proposition \ref{lemma_Control 1}, for $\varepsilon>0$ and  $h_\alpha\in Y,$ there exists a boundary source $a_\varepsilon=a_{\varepsilon,\alpha}$, for which 
\ba
\| u^{Ph_\alpha}(T)-\p_t u^{a_\varepsilon}(T)\|_{L^2(M)}^2 +\| u^{a_\varepsilon}(T)\|_{H^1(M)}^2< \varepsilon.
\ea
On the other hand for every $\beta \in (0,1)$  the minimizer $a(\beta, h_\alpha)$ satisfies
\ba
\| u^{Ph_\alpha}(T)-\p_t u^{{a(\beta, h_\alpha)}}(T)\|_{L^2(M)}^2 +\| u^{{a(\beta, h_\alpha)}}(T)\|_{H^1(M)}^2+\beta\|{a(\beta, h_\alpha)}\|^2_{Y}\\
\le
\| u^{Ph_\alpha}(T)-\p_t u^{a_\varepsilon}(T)\|_{L^2(M)}^2 +\| u^{a_\varepsilon}(T)\|_{H^1(M)}^2+\beta\|a_\varepsilon\|^2_{Y}.
\ea
We choose $\beta\le\frac{\varepsilon}{1+\|a_\varepsilon\|^2_{Y}}$, and thus
\ba
\| u^{Ph_\alpha}(T)-\p_t u^{{a(\beta, h_\alpha)}}(T)\|_{L^2(M)}^2 +\| u^{{a(\beta, h_\alpha)}}(T)\|_{H^1(M)}^2 \le 2\varepsilon,
\ea
and we see that $\p_t u^{{a(\beta, h_\alpha)}}(T)\to u^{Ph_\alpha} (T)$ in $L^2(M)$ and $u^{{a(\beta, h_\alpha)}}(T)\to 0$ in $H_0^1(M)$, as $\beta \to 0$.
This and (\ref{sikaniska}) yield the claim.

\end{proof}

\section{Focusing of waves}\label{sectionFOCUS}

In this section we  prove Theorem \ref{Thm 1}. 
\medskip

\noindent{\bf Notation 1.}
Let $T> 2\operatorname{diam}(M)$, let $\widehat{x}=\gamma_{\widehat
  z,\nu}({{{{{{\widehat t}}}}}})$, where $\widehat z\in {\p M}$, and $0<{{{{{{\widehat t}}}}}}<T$.
Let $\Gamma_k\subset {\p M}$ for $k=1,2,\dots$ be open neighborhoods of
$\widehat z$, such that $\hbox{diam}(\Gamma_k)<1/k$, $\Gamma_k \supset \overline \Gamma_{k+1}$
and $\bigcap_{k=1}^\infty\Gamma_k=\{\widehat z\}$.
\medskip

Let $a{{(\alpha,\beta}},k), \tilde a{{(\alpha,\beta}}, k) \in Y$ be functions described in Lemma \ref{finalLemma}, with the corresponding sets $\mathcal B\subset {\p M}\times \R_+$ of the form
\begin{eqnarray}
\label{set1}
  B(k) = {\p M}\times \left(T-(\widehat{t}-\frac{1}{k}), T\right),\,k\in\mathbb N,
\end{eqnarray}
and 
\begin{align}
\label{set2}
\tilde B(k) = \left( {\p M}\times \left(T-(\widehat{t} - \frac{1}{k}),T\right) \right) \,\cup\, \left(\Gamma_k \times \left(T-\widehat{t},T \right)\right),
\end{align}
respectively, where $k \in\mathbb N$.
Under these assumptions, we define
\begin{align}
\label{jonob}
  b{{(\alpha,\beta}},k) =
\tilde a{{(\alpha,\beta}}, k) - a{{(\alpha,\beta}},k)
\in Y.
\end{align}


By Lemma \ref{finalLemma},  in the space $H_0^1(M)\times L^2(M)$ we have the limits
\begin{eqnarray}\label{eq: A1}
  \lim_{\alpha\to 0}\lim_{\beta\to 0}\lim_{n\to \infty}
  \vc{  u^{a_n{{(\alpha,\beta}},k)}(T)}{u^{a_n{{(\alpha,\beta}},k)}_t(T)} &=& \vc{ 0}{\mathds{1}_{\mathcal N(k)}}, \\
\label{eq: A2}
  \lim_{\alpha\to 0}\lim_{\beta\to 0}\lim_{n\to \infty}
  \vc{  u^{\tilde a_n{{(\alpha,\beta}}, j,k)}(T)}{u^{\tilde a_n{{(\alpha,\beta}}, k)}_t(T)} &=& \vc{0}{\mathds{1}_{\tilde{\mathcal N}(k)}},
\end{eqnarray}
where
$$
 \mathcal N(k) = M({\p M},{{{{{{\widehat t}}}}}}-\frac{1}{k}),  \quad\quad
  \tilde{\mathcal N}(k) = M({\p M},{{{{{{\widehat t}}}}}}-\frac{1}{k}) \cup M(\Gamma_k, \widehat{t}).
$$
Let $k \in\mathbb N$ and we define 
\begin{eqnarray}\label{jeps}
  \Omega_{k} = \tilde{\mathcal N}(k) \setminus \mathcal N(k). 
\end{eqnarray}

\begin{proof}[Proof of Theorem \ref{Thm 1}] 

As $ \Omega_{k+1}\subset  \Omega_{k}$ and $ \Omega_{k}\subset  
 M(\hat z,{{{{{{\widehat t}}}}}}+\frac{1}{k})\setminus
M({\p M},{{{{{{\widehat t}}}}}}-\frac{1}{k}),$ it follows from  \cite[Lemma 12]{DKL},
that if ${{{{\widehat t}}}}<\tau_{\p M}(\widehat z)$ then
$\bigcap_{k=1}^\infty\Omega_{k}=\{\hat x\}$ where 
$\widehat{x}=\gamma_{\widehat
  z,\nu}({{{{\widehat t}}}})$. If ${{{{\widehat t}}}}>\tau_{\p M}(\widehat z)$ 
 then $\bigcap_{k=1}^\infty\Omega_{k}=\emptyset.$

Lemma \ref{finalLemma} and Theorem \ref{Main Fuu} 
imply that the 
 boundary sources 
$a{{(\alpha,\beta}},k), \tilde a{{(\alpha,\beta}}, k) \in Y$  described in Lemma \ref{finalLemma}
and $a{{(\alpha,\beta}},k)$ given in
\eqref{jonob}
 satisfy  in the space $H_0^1(M)\times L^2(M)$ the limit
\beq\label{limit in lemma}
\lim_{\alpha\to 0}\lim_{\beta\to 0}\lim_{n\to \infty}
\vc{  u^{b_n{{(\alpha,\beta}},j,k)}(T)}{u^{b_n{{(\alpha,\beta}},k)}_t(T)} =
\vc{0}{\mathds{1}_{\Omega_{k}}},
\eeq
where $\Omega_{k}$ is defined in (\ref{jeps}).
The volumes of the sets $\Omega_{k}$  can be written as the inner products,
\[
\underset{n\to\infty}{\lim}\bra \pat b_n{{(\alpha,\beta}},k), \Phi_T\cet=\underset{n\to\infty}{\lim}\bra u_t^{b_n{{(\alpha,\beta}},k)}(T), 1\cet_{L^2(M)}=\vol (\Omega_{k})
\] and hence we can also determine $\vol (\Omega_{k})$ using the map $\Lambda$. 
Thus we can  define
\beq\label{fb formula}
f_n{{(\alpha,\beta}},k)=\frac 1{\vol (\Omega_{k}) }b_n{{(\alpha,\beta}},k),
\eeq
and we are ready to prove the the main result of this paper.

Below, $X'$ is the dual space of $X$ with respect to the pairing defined by the $L^2$-inner product of the distributions and test functions. 
Let 
$X^{s}= \mathcal D((1-\mathcal A)^{\frac{s}{2}})\subset H^s(M),$ $s\ge 0$ be the domain of the $s$-th power of the selfadjoint operator $(1-\mathcal A)$
endowed with the Neumann boundary values and let $X^{-s}$ denote the dual space of $X^{s}$. Note that as  $H_0^{s}(M)\subset  X^s$, for $s>0$,
we have that the embedding
$ X^{-s}\subset H^{-s}(M)$ is continuous.


We have $\lim_{k\to \infty} (\lim_{\alpha\to 0}\lim_{\beta\to 0}\lim_{n\to \infty}{u^{b_n{{(\alpha,\beta}},k)}(T)} )=0$ in $H^1_0(M)$ and
\ba
& &\lim_{k\to \infty}(\lim_{\alpha\to 0}\lim_{\beta\to 0}\lim_{n\to \infty}
{u^{b_n{{(\alpha,\beta}},k)}_t(T)} )=
\lim_{k\to \infty}
\frac{\mathds{1}_{\Omega_k}(x)}{\vol(\Omega_k )}
=\delta_{\hat y}(x)
\ea
in $C(M)'\subset ({H}^{s}(M))'\subset \mathcal D ((1-\mathcal A)^{-\frac{s}{2}})
\subset {H}^{-s}(M),$ where $s>\frac{\dim(M)}{2}=\frac{n}{2}$.   Thus the claims of Theorem \ref{Thm 1} follow from  
and formulas \eqref{limit in lemma} and (\ref{fb formula}). \end{proof}



\begin{lemma} Let $T_1>T>\diam(M)$.
For $\widehat z\in {\p M}$ and ${{{{\widehat t}}}}<\tau_{\p M}(\widehat z)$ and the point  $\widehat x=\gamma_{\widehat z,\nu}({{{{\widehat t}}}})\in M$ we have 
\beq
\label{eq: A0B}
 \lim_{k\to \infty}\lim_{\alpha\to 0^+}\lim_{\beta\to 0^+}\lim_{n\to \infty}
 u^{f_n{{(\alpha,\beta}},k)} \bigg|_{\p M\times (T,T_1)}=
G(\,\cdotp,\,\cdotp;\hat x,T)\bigg|_{\p M\times (T,T_1)},\hspace{-1cm}
\eeq 
where the limit takes place in $\left(H^s_0(\p M\times (T,T_1)\right))'$, $s>\dim(M)/2$.
Moreover, if ${{{{\widehat t}}}}>\tau_{\p M}(\widehat z)$ the above limit is zero.

\end{lemma}

 \begin{proof} We will show that we can define boundary values (or the trace) of both
 sides of equation \eqref{eq: A0}. 
To this end, consider the map $W:(\phi_0,\phi_1)\to u|_{\p M\times (T,T_1)}$,
where $T_1>T$ and 
\begin{align}
\label{eq: Wave 2}
\begin{cases}
  \p_t^2u(x,t)+\A u(x,t)=0,\quad \hbox{ in } M\times (T,T_1),\\
  u|_{t=T}=\phi_0,\quad u_t|_{t=T}=\phi_1,  \
  \p_\nu u|_{\p M \times (T,\infty)}=0.
\end{cases}
\end{align}
The map $W:H^1_0(M)\times L^2(M)\to L^2\left(\p M\times (T,T_1)\right)$ is bounded,
 its adjoint is the map $W^*:h\mapsto (\p_t w|_{t=T},w_t|_{t=T})$
where $w$ is the solution of the time-reversed wave equation with the Dirichlet boundary value,
\begin{align*}
\begin{cases}
   \p_t^2w(x,t)+\A w(x,t)=0,\quad \hbox{ in } M\times (T,T_1),\\
  w|_{t=T_1}=0,\quad w_t|_{t=T_1}=0,  \
  w|_{\p M \times (T,T_1)}=h .
\end{cases}
\end{align*}
The map $W^*:  L^2\left(\p M\times (T,T_1)\right)\to \left(H^1_0(M)\times L^2(M)\right)'$
is continuous (see \cite{KKL}, Lemma 2.42). Also, the restriction of the map $W^*$ to a smoother Sobolev spaces,  $W^*:  H^s_0\left(\p M\times (T,T_1)\right)\to H^s_0(M)\times H^{s+1}_0(M),$ $s>0$ is continuous by \cite{KKL}, Theorem 2.46. This implies that  the map $W$ has a continuous extension 
$W:\left(H^{s}_0(M)\times H^{s+1}_0(\M)\right)'\to \left(H^s_0(\p M\times (T,T_1))\right)'$.
We can  use this to define the Dirichlet boundary value for a non-smooth solution of a Neumann problem in the weak sense and we define $$u|_{\p M \times (T,T_1)}=W(\phi_0,\phi_1)$$ for a solution $u$ of \eqref{eq: Wave 2} with $(\phi_0,\phi_1)\in \left(H^{s}_0(M)\times H^{s+1}_0(M)\right)'$.  As the map $W:\left(H^{s}_0(M)\times H^{s+1}_0(\M)\right)'\to \left(H^s_0(\p M\times (T,T_1))\right)'$  is continuous, we obtain \eqref{eq: A0B} from the limit
 \eqref{eq: A0}.

\end{proof}

 Using methods developed in \cite{BelishevBasis} we next consider a special case of an isotropic, or, a conformally Euclidean metric 
 \begin{lemma}\label{lemmaLast} 
Assume that 
$M\subset \R^m$ and 
the operator $\mathcal A$ is of the form $\mathcal A=-c(x)^2\Delta$. Then for $w_j(x)=x_j$ we have
\beq\label{coordinate inner product}
\langle u^f(x,T), w_j\rangle_{L^2(M)}=
\langle \Lambda^* (\Phi_t\frac{\partial w_j}{\partial \nu})-\Phi_T w_j ,f\rangle.
\eeq
\end{lemma}

\begin{proof}
As $w_j(x)=x_j$ satisfies $\mathcal Aw_j=0$, the inner product 
$$
I_j(t)=\int_M u^f(x,t)w_j(x) c(x)^{-2}dx
$$
satisfies the initial boundary value problem 
\ba
\p_t^2 I_j(t)=\int_{\p M}((\Lambda f)\frac{\partial w_j}{\partial \nu}-f w_j)dS(x),\quad
\p_t I_j(t)|_{t=0}=0,\,\,\,
I_j(t)|_{t=0}=0.
\ea
By solving this ordinary differential equation we obtain \eqref{coordinate inner product}.
\end{proof}

Lemma \ref{lemmaLast}  implies that
when  the operator $\mathcal A$ has the form $\mathcal A=-c(x)^2\Delta$,
the coordinates of the point $\hat x$ where the waves focus can be computed a posteriori.

\begin{corollary}\label{cor: focusing coordinates}
Assume that $M\subset \R^m$ and the operator $\mathcal A$ is of the form $\mathcal A=-c(x)^2\Delta$. 
Let $\widehat z\in {\p M}$, $\widehat x=\gamma_{\widehat z,\nu}({{{{\widehat t}}}})\in M$,
$0<{{{{\widehat t}}}}<T$ and let 
$f_n{{(\alpha,\beta}},k)$ be the sources defined in 
 Theorem \ref{Thm 1}. Then the Euclidean coordinates of the point $\hat x=(\hat x_\ell)_{\ell=1}^m \in \R^m$ are given by
\[
\underset{k\to \infty}{\lim}\left(\underset{\alpha\to 0}\lim
\underset{\beta\to 0}
{\lim} \underset{n\to\infty}{\lim}
\frac{\langle u^{f_n{{(\alpha,\beta}},k)},w_\ell\rangle_{L^2(M)}}{\langle u^{f_n{{(\alpha,\beta}},k)},1\rangle_{L^2(M)}}\right)
=\hat x_\ell,
\]
where the inner products on the left hand side are determined by
$\Lambda$ via the formulas \eqref{Tformula2} and \eqref{coordinate inner product}.
\end{corollary}


\section{Construction of boundary sources sources via iterated measurements}\label{sectionITERATION}
In this section we present a modified time-reversal iteration scheme for determination of the boundary sources.
$h_\alpha$ and $ a{{(\alpha,\beta}})$ given in \eqref{minimizer a} and
\eqref{minimizer ab}, respectively. We explain this in a general framework.

 Let $H$ be Hilbert space and let $L:H\to H$ be linear, non-negative selfadjoint operator. Let $\alpha\in(0,1)$ and $f\in H$. Then there is a solution $g_\alpha$ for problem
\begin{equation}
\label{karju}
\left(
L+\a
\right)g_\alpha=f.
\end{equation}
Let $\omega>0$ be such that $\omega>2(1+\Vert L \Vert_H)$, and let
\begin{eqnarray}
\label{Sdef}
  S=(1-\frac {\alpha}{\omega})I-\frac {1}{\omega}L.
\end{eqnarray}
Then \eqref{karju} is equivalent to $(I-S)g_\alpha=\displaystyle \frac{1}{\omega}f.$

We define a sequence $g_n\in
Z$, $n=1,2,\ldots$ by
\begin{eqnarray}
\label{firstItScheme}
  g_0(\alpha) = \displaystyle \frac{1}{\omega}f, \quad
  g_n (\alpha)= g_0(\alpha) + Sg_{n-1}(\alpha), \quad n=1,2,\ldots.
\end{eqnarray}

\sloppy{
\begin{theorem}[Iteration of boundary sources] Let $g_\alpha$ be defined by (\ref{karju}) and let the sequence $g_1(\alpha), g_2(\alpha),\ldots$ be defined by (\ref{firstItScheme}).  Then 
$
  \lim_{n\to \infty} g_{n}(\alpha)  = g_\alpha
$ in the space $H$. \label{Main Fuu}
\end{theorem}}

\begin{proof} 
 Since  operator $L$ is a positive operator satisfying $0\le L\le \|L\| I$ and $\frac 1  \omega (\alpha+\Vert L \Vert)<\frac 12$, we see using spectral theory and  (\ref{Sdef}) that  $\frac 12 I\le S\le (1-\frac {\alpha}{\omega})I$.
Hence  $\Vert S\Vert<1$.  
Thus we see using the  Neumann series  that 
 \begin{eqnarray*}
g_\alpha = (I-S)^{-1}\Big(\frac {f}{\omega} \Big ) =\underset{n=0}{\overset{\infty}{\sum}} S^n \Big(\frac {f}{\omega} \Big )=  \lim_{n\to \infty} g_{n}(\alpha).
\end{eqnarray*}

\end{proof}

To obtain the boundary sources $h_\alpha$ and $ a{{(\alpha,\beta}})$
that produce the focusing waves we apply Theorem \ref{Main Fuu} in the two cases: 
To obtain $h_\alpha$ we consider the setting of
Theorem \ref{min_theorem_h} where
 the Hilbert space $H$  is $V$, the operator $L$ is defined by $$L=PKP\quad\hbox{and}
\quad f=P\Phi_T.$$
 To obtain $ a{{(\alpha,\beta}})$  we consider the setting of
 Theorem \ref{min_theorem_a} where the Hilbert space $H$  is $Y$, the operator $L$ is defined by $$L={N_Y}Q\,\left(R\Lambda R \p_t\hat P-\hat P\p_t\Lambda +K\right)
 \quad\hbox{and}
\quad f=-{N_Y}Q\p_t KPh_\a.$$
In these cases, we call the iteration \eqref{firstItScheme} the \emph{modified time reversal iteration} scheme as 
in the iteration \eqref{firstItScheme}
we iterate  simple operators, such as ${N_Y},Q,\hat P$ and the time-reversal operator $R$, and the measurement operator $\Lambda$. In particular, the iteration \eqref{firstItScheme} can be implemented in an adaptive way,
where we do not make physical measurements to obtain the complete operator $\Lambda$, but evaluate the operator $\Lambda$ 
only for the boundary sources appearing in the iteration. In other words, we do not make measures to obtain the whole operator
(or ``matrix'') $\Lambda$ but make a measurement only when the operator $\Lambda$ is called in the iteration.
By doing this, the effect of the measurement errors is reduced as in each step of the
iteration, the measurement errors are independent. This strategy to do imaging using iteration of Neumann-to-Dirichlet map originates
from works of Cheney, Isaacson, and Newell \cite{Cheney,Isaacson}, see also \cite{CIL} the applications for acoustic measurements.
%


\section{Computational study in $1+1$ dimensions} 
\label{sec_computations}
In this section we present a computational implementation of our energy focusing method for a $1 + 1$-dimensional wave equation.
Let $M$ be the half axis $M=[0,\infty)\subset \R$, $T > 0$ 
and consider 
the Neumann-to-Dirichlet operator $ \Lambda= \Lambda_c$,
\begin{align*}
 \Lambda : L^{2}(0,2T) \to L^{2}(0,2T),\quad
 \Lambda f=u^f|_{x=0},
\end{align*}
where $u$ is the solution of 
\begin{align}
\label{dartwader}
&\left(\frac {\p^2}{\p t^2} - c(x)^2 \frac {\p^2}{\p x^2}\right) u(x,t) = 0 \quad \text{in $M\times (0,2T)$},
\\\nonumber
&\p_x u(0,t) = f(t),
\quad u|_{t = 0} =0,\quad  \p_t u|_{t=0} = 0.
\end{align}
We assume that 
    \begin{align}\label{c_bounds}
C_0\le c(x)\le C_1, \quad \supp(c-1) \subset (L_0, L_1)
    \end{align}
for some $0 < C_0 < C_1$ and $0 < L_0 < L_1$.
In order to be able to control $u(x,T)$ for $x \in (L_0,L_1)$
using $f$ in the sense of Proposition \ref{dense1}, we assume furthermore that 
\begin{align}
\label{aikapoika}
T> \frac{L_1}{C_0}.
\end{align}


\begin{figure}
\centering
\includegraphics[scale=0.09]{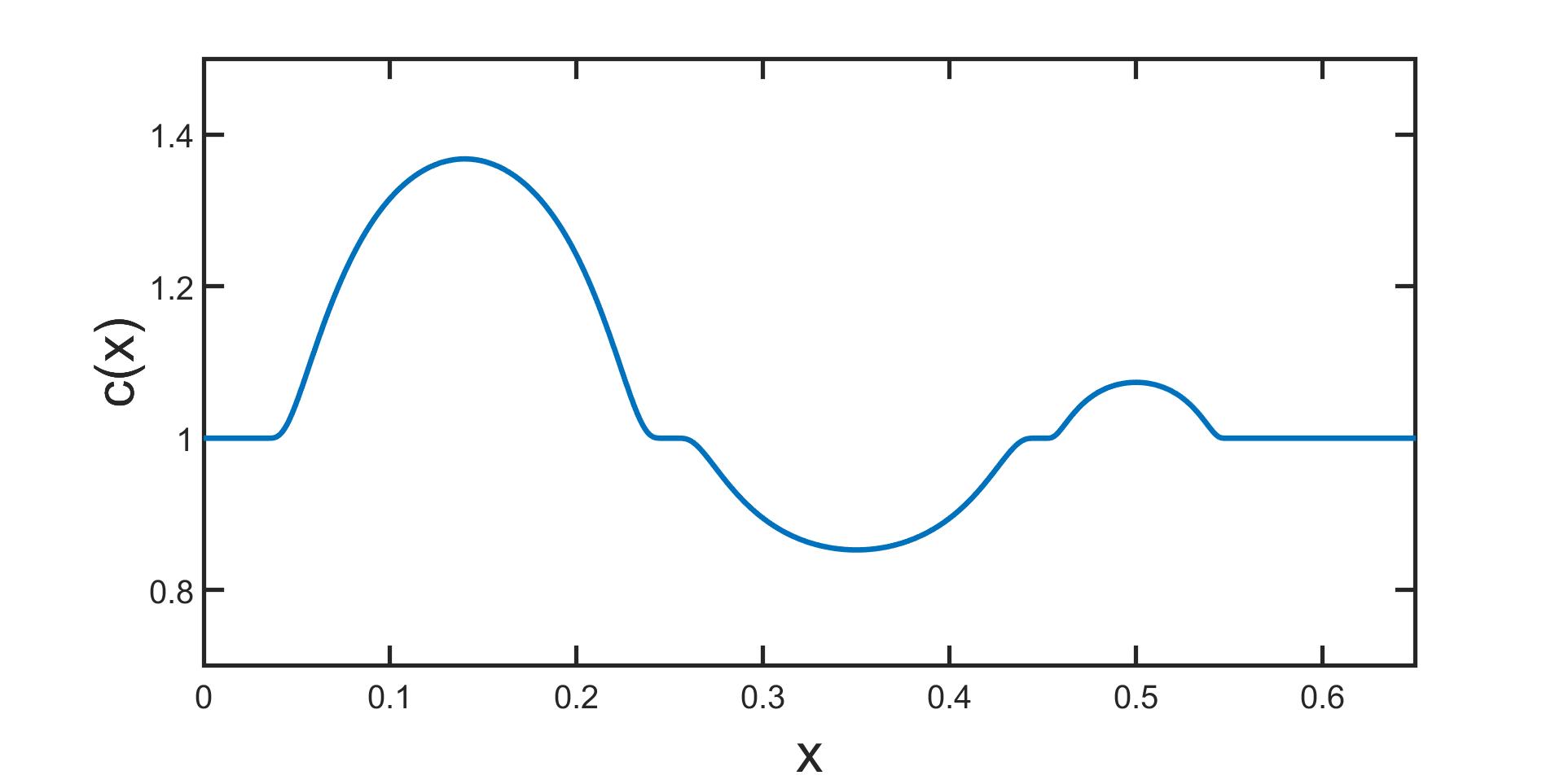}\quad \includegraphics[scale=0.12]{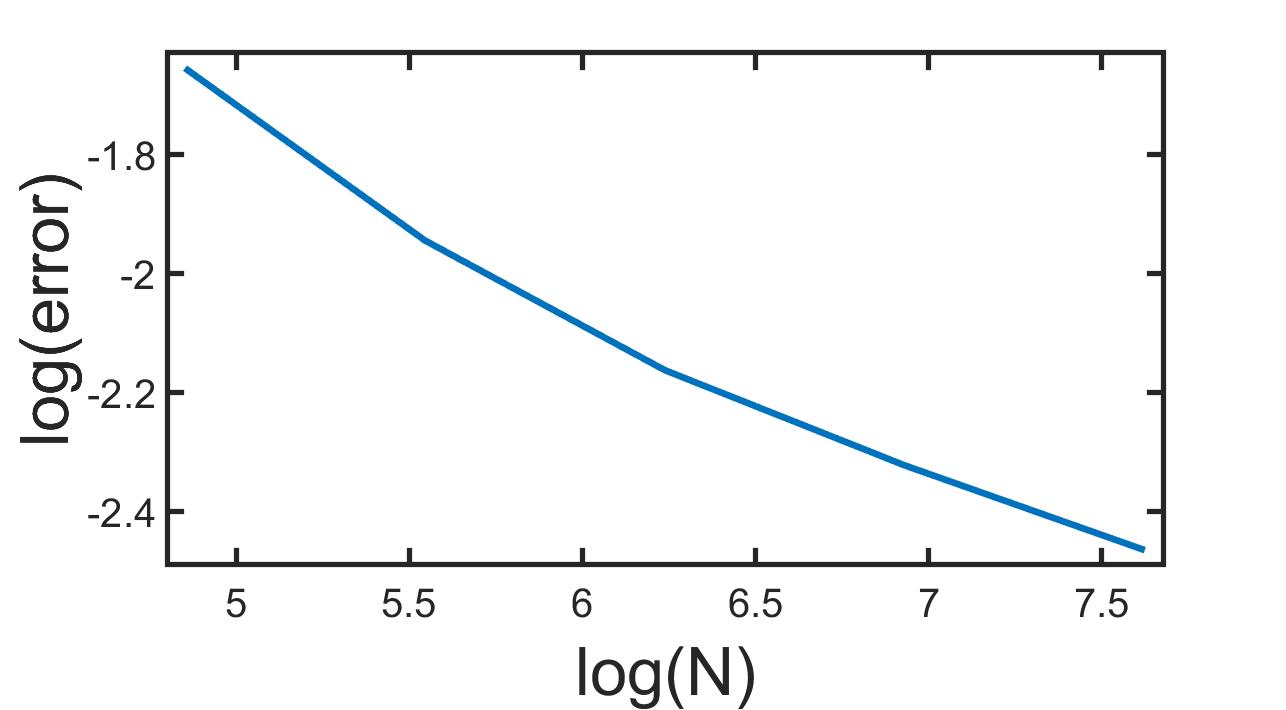}
\caption{Left: The function $c(x)$ used in computational examples. Right: Convergence of the error (\ref{error}) as a function of $\bold{N}$ in log-log axes.
}
\label{fig:sikapossu12}    
\end{figure}
We use the wave speed function $c$ in Figure \ref{fig:sikapossu12}
in all the computational examples below. 
It satisfies the bounds (\ref{c_bounds}) with $L_0 = 0.05$, $L_1 = 0.55$, $C_0 = 0.8$ and
$C_1 = 1.4$. Moreover, we take $T=2$ and then (\ref{aikapoika}) holds.
In the one dimensional case, the travel time metric is given by metric tensor
$g=c(x)^{-2}dx^2$ and the corresponding distance function $d(x_1, x_2) =d_g(x_1, x_2) $ (i.e., travel time beween
points is given by
\beq
d(x_1, x_2) = \int_{x_1}^{x_2} \frac{1}{c(x)} dx,\quad x_1\leq x_2.
\eeq 
We denote by ${\bf x}(r)$ the point that satisfies $d(0,{\bf x}(r))=r$, that is, ${\bf x}(r)\in M$
is the point which travel time to the boundary point 0 is $r$. 
The domain of influence for the boundary point $0$ and time $r>0$ are
\begin{align}
\label{Mr}
M(r) = \{ x \in M; d(0,x) \le r \}.
 \end{align}

\subsection{Simulation of measurement data}

We use $H^1$-conformal piecewise affine finite elements on a regular grid on $(0,2T)$ to discretize the Neumann-to-Dirichlet operator $\Lambda$. 
Let us explain this in more detail. For $\bold{N}\in \zeta$ and $n =1,...,2\bold{N}-1$ we write $h = T/\bold{N}$ and denote by $\phi_{n,\bold{N}} \in H_0^1(0,2T)$ the function that is supported on $[(n-1) h, (n+1)h]$, that 
satisfies $\phi_{n,\bold{N}}(nh) = 1$, and whose restrictions on 
$[(n-1) h, nh]$ and $[nh, (n+1)h]$ are affine. 
Then the subspace 
\begin{align}
\label{pcsfjoukko}
\mathcal P^\bold{N}=\linspan\big\{\phi_{1,\bold{N}},\dots,\phi_{2\bold{N}-1,\bold{N}}\big\}\subset H^1_0(0,2T) 
\end{align}
consists of piecewise affine functions and we write 
\begin{align}
\label{pcsfjoukko2}
P^{\bold N}:H^1_0(0,2T)\to\mathcal P^{\bold N},\quad P^{\bold N} f(t) =\sum_{j=1}^{2{\bold N}-1} f(jh)\phi_{j,\bold N}(t).
\end{align}
for the corresponding interpolation operator. 
The function $u^{\phi_{1,\bold{N}}}$, solving (\ref{dartwader}) with $f=\phi_{1,\bold{N}}$, is computed with high accuracy using the $k$-Wave solver \cite{Treeby}. Then we define the discretization of $\Lambda$,
    \begin{align*}
\Lambda_{\bold{N}}^{(d)} : \mathcal P^{\bold{N}} \to \mathcal P^{\bold{N}},
    \end{align*}
by 
$\Lambda_{\bold{N}}^{(d)} \phi_{1,\bold{N}} = P^\bold{N} (u^{\phi_{1,\bold{N}}}|_{x=0})$
together with the translation invariance in time,
$\Lambda_{\bold{N}}^{(d)} \phi_{j,\bold{N}}(t) = \Lambda_{\bold{N}}^{(d)} \phi_{1,\bold{N}}(t-(j-1)h),$ for $j=2,3,\dots,2\bold{N}-1$.
We can also write 
    \begin{align*}
\Lambda_{\bold{N}}^{(d)}f=\sum_{j=1}^{2\bold{N}-1}\sum_{k=1}^{j}f_k\Lambda_{j-k+1}\phi_{j,\bold{N}}, \quad \text{for} \quad f = \sum_{j=1}^{2\bold{N}-1} f_j \phi_{j,\bold{N}}.
    \end{align*}
In the computational examples, $u^{\phi_{1,\bold{N}}}$ is solved using a regular mesh with $2^{13}$ spatial and $2^{15}$ temporal cells.

\subsection{Implementation of the energy focusing}

Computational implementation of the energy focusing method
boils down to solving discretized versions of the linear equations 
(\ref{eq1}) and (\ref{eq2}).

Most of the operators $X$ appearing in (\ref{eq1}) and (\ref{eq2}) are simply discretized by setting $X^{(d)} \phi_{j,\bold{N}} = P^\bold{N} X \phi_{j,\bold{N}}$. This is the case for $R$ and $J$, see the definition (\ref{K-operator}) of $K$, as well as, for $N$ and $Q$ in (\ref{eq2}).

In the $1+1$-dimensional case, the projection $P$, appearing in (\ref{eq1}) and (\ref{eq2}), is equal to the multiplication with the characteristic function of the interval $(T-r,T)$ for some $r$,
that is, $$P=P_r:L^2(0,T)\to L^2(0,T),\quad (P_rf)(t)=\mathds{1}_{(T-r, r)}(t)f(t).$$
We discretize $P$ by setting
    \begin{align*}
P^{(d)} \phi_{j,\bold{N}}= \begin{cases}
\phi_{j,\bold{N}}, & T-r < jh < T,
\\
0, & \text{otherwise}.
\end{cases}
    \end{align*}
Then $P^{(d)} : \mathcal P^{\bold{N}} \to \mathcal P^{\bold{N}}$.
The projection $\hat P$ is discretized analogously, see the definition (\ref{L-oper}) of $L$.
The time derivative is discretized using first order forward finite differences at the points $nh$, $n=1,\dots,2\bold{N}-2$, as follows
    \begin{align*}
\p_t^{(d)} f = \sum_{j=1}^{2\bold{N}-2} \frac{f_{j+1}-f_j} h \phi_{j,\bold{N}}(t), 
\quad \text{for} \quad f = \sum_{j=1}^{2\bold{N}-1} f_j \phi_{j,\bold{N}}.
    \end{align*}

We have now given discretizations of all the operators appearing in (\ref{eq1}) and (\ref{eq2}).
The function $\Phi_T$ on the right-hand side of (\ref{eq1})
is discretized by $\Phi_T^{(d)} = P^\bold{N} \Phi_T$.
Solving for $h$ in (\ref{eq1}), with the operators replaced by their discretizations, gives us $h_\alpha^{(d)} \in \mathcal P^\bold{N}$.
Then we solve for $a$ in (\ref{eq2}), with the operators replaced again by their discretizations, and with $h_\alpha$ replaced by $h_\alpha^{(d)}$. We denote the so obtained solution by $a^{(d)} \in \mathcal P^\bold{N}$.

We use the restarted generalized minimal residual (GMRES) method to solve the discrete versions of (\ref{eq1}) and (\ref{eq2}). 
The maximum number of outer iterations is 6 and the number of inner iterations (restarts) is 10. We use zero as the initial guess, and the tolerance of the method is set to $10^{-12}$.

\begin{figure}[H]
\vspace{-5mm}

    \centering
    \includegraphics[scale=0.13]{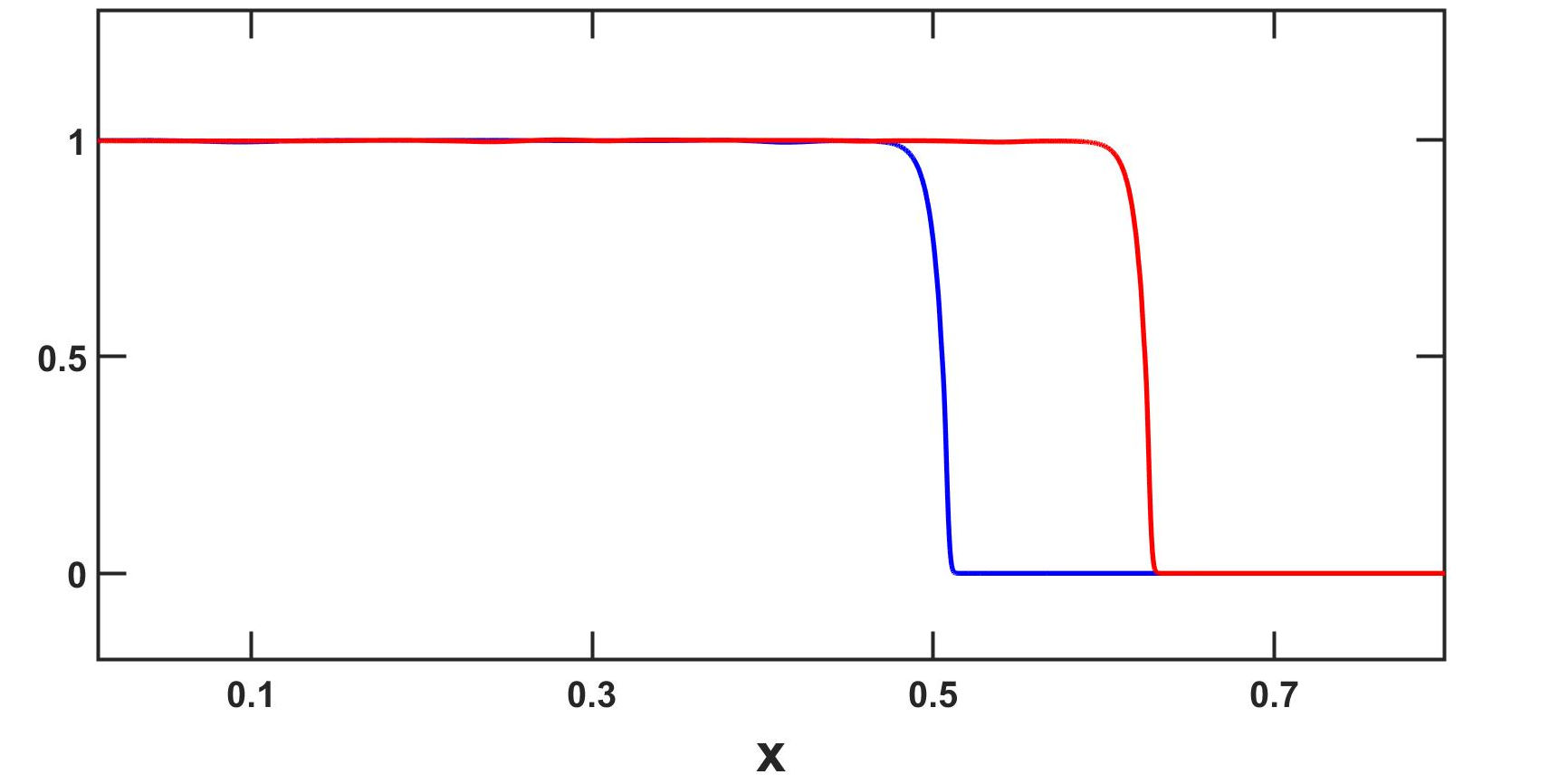}
    
    \vspace{-5mm}
    \caption{
Functions
$u^{P_{r_1}  h_1} (x, T ) \approx 1_{M (ˆ'r_1)}(x)$ (blue) and $u^{P_{r_2}
h_2} (x, T ) \approx 1_{M (r_2)}(x)$ (red),
where $h_1$ and $h_2$ are obtained by solving the discretized version
of (\ref{eq1}).}
\label{fig_h_sol}
\end{figure}

\subsection{Computational examples}

We set $r_1 = \frac 12 $, $r_2 = \frac 58$ and $\bold{N} = 2^{11}$, and denote by $h_{\alpha,j}^{(d)}$ the solution of the discretized version of (\ref{eq1}) with $r=r_j$, $j=1,2$. 
The solutions $u^{P_{r_1}
h_{\alpha, 1}^{(d)}} (x, T )$ and $u^{P_{r_2}  h_{\alpha, 2}^{(d)}} (x, T )$ with $\alpha = 0.001$ are shown in Figure~\ref{fig_h_sol}.
Moreover, we denote by $a_{j}^{(d)}$ the solution of the discretized version of (\ref{eq2}) with $h_\alpha = h_{\alpha,j}^{(d)}$.
The difference of the corresponding solutions 
    \begin{align*}
u^{a_{2}^{(d)}}(x,T)-u^{a_{1}^{(d)}}(x,T) =
u^{a_{2}^{(d)}-a_{1}^{(d)}}(x,T),
    \end{align*}
with $\beta=1.02\,\cdotp 10^{-4}$ and $\alpha$ as above, is shown in Figure \ref{fig:sikapossu1212}.
The spurious oscillations near the origin in Figure \ref{fig:sikapossu1212} were present also in our computations using finer discretizations, however, they appear to converge to zero in $L^2(M)$
as predicted by Theorem \ref{Thm 1}.
Convergence of the error
    \begin{align}\label{error}
\|u^{a_{2}^{(d)}-a_{1}^{(d)}}_t(\cdot,T) - \mathds{1}_{M(r_2)\setminus M(r_1)}\|_{L^2(M)}
    \end{align}
is shown in Figure \ref{fig:sikapossu12} (Right)  as a function of $\bold{N}$. 
Different regularization parameters $\alpha$ and $\beta$ are chosen for each $\bold{N}$.


\begin{figure}[H]
\vspace{-3mm}

    \centering
    \includegraphics[scale=0.18]{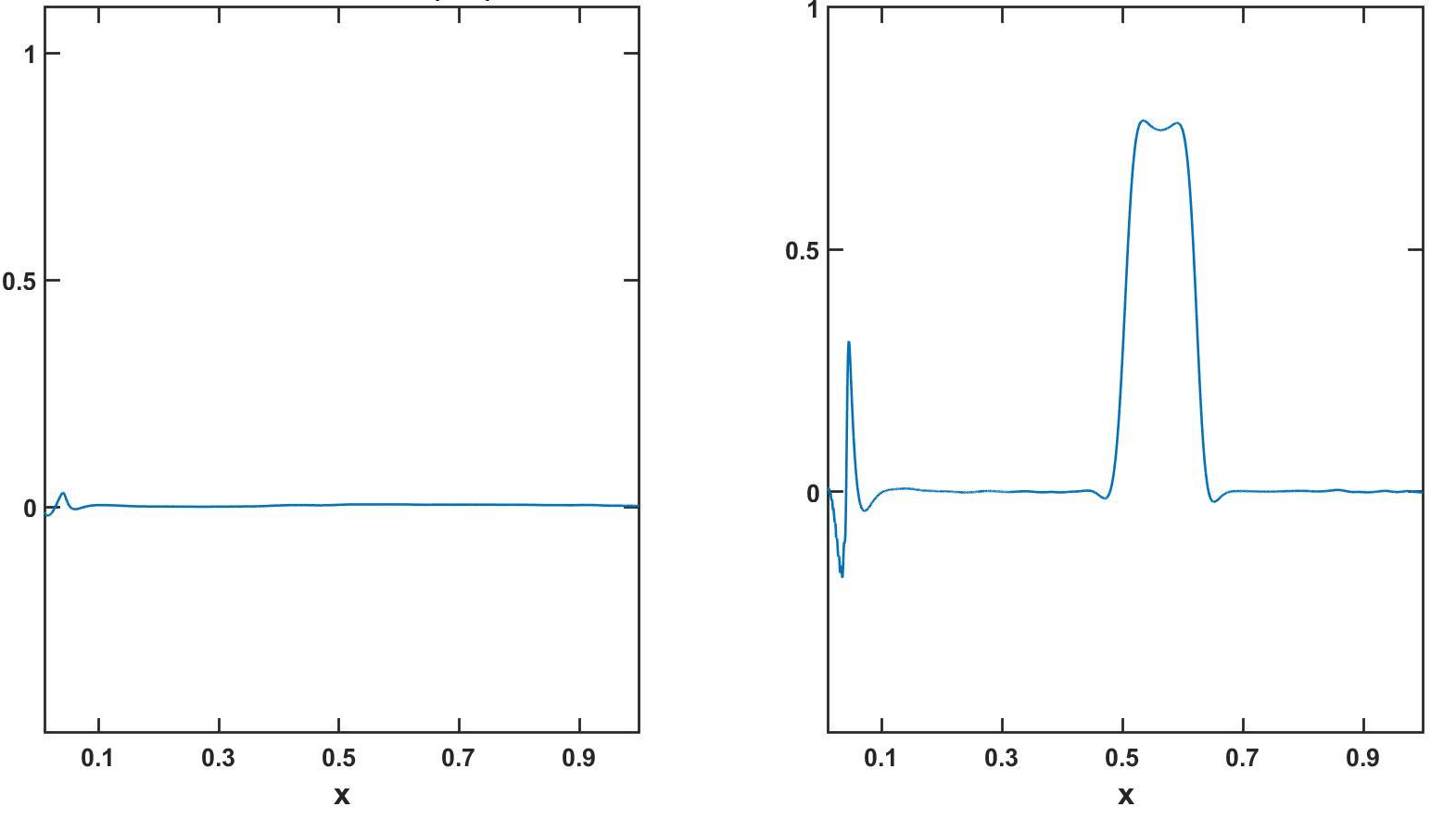}
    
    \vspace{-5mm}
    \caption{
Functions
$u^{a} (x, T )$ (left) and $\p_t u^{a} (x, T )$ (right)
where $a=a_{2}^{(d)}-a_{1}^{(d)}$ and
 $a_1^{(d)}$ and $a_2^{(d)}$ are the solutions to the discretized version of the
minimization problem (\ref{eq2}).
 The time derivative of the wave at time $T$, that is,
$x\mapsto u^{a}_t(x,T)$, where $a=a_{2}^{(d)}-a_{1}^{(d)}$, is concentrated near the 
interval $[{\bf x}(r_1),{\bf x}(r_2)]=\hbox{cl}({M(r_2)\setminus M(r_1)})$,
where ${\bf x}(r_1)\approx 0.5$ and ${\bf x}(r_2)\approx 0.62.$
The ``spike'' in the time derivative on right close to
the value x = 0.05 has a relatively small $L^2$-norm despite its
visual appearance.
}
    \label{fig:sikapossu1212}
    \end{figure}
\section{Observation times  and boundary distance functions}\label{sectionDDF}
%

In this section we will 
apply focusing of waves 
to  inverse problems, that is, to determine the coefficients of the operator $\A$ that correspond to the unknown material functions in $M$.  Results in \cite{B1,BK,KKL}  show that  
the mapping $\Lambda$ determines uniquely the
isometry type of the Riemannian manifold $(M,g)$. Here we consider an alternative proof for these results. We show that $\Lambda$ determines the time when the wave emanating from a point source in the domain $M$  is observed at different points of the boundary $\p M$. We do this by
considering waves that focus at a point $\hat x$. As shown in formula \eqref{eq: A0 Green},
the waves focusing at time $T$  to the point $\hat x$ converge to Green's function
$G(x,t,\hat x,T)$ at times $t>T$. Below we show that by considering the boundary values
of the focusing waves we can determine the observation times from  point sources
located at all points $\hat x\in M$. These functions determine the metric $g$  in $M$ up to an isometry,
see \cite{KKL}. A similar approach has been used for non-linear wave equation, e.g. $\square_gu+au^2=0$,
where the non-linear interaction of the waves is used to produce artificial microlocal points sources in $M\times \R$.
Such  artificial microlocal points sources determine the information analogous 
to the observation times from point sources in the medium, see 
\cite{deHoop_et_al,Hintz,KLUinv,Lassas-ICM}. We note that for genuinely non-linear equations this technique makes it possible 
to solve inverse problems for non-linear equations that are still unsolved for linear equations, e.g. for equations with a time-varying metric.
Below, we will show that some of these techniques are applicable also for linear wave equations.

%

Consider a manifold $(M,\texttt{g})$ and Green's function $G(\cdot, \cdot\,; \hat{x}, T)$ satisfying \eqref{GreensWaveEq}.
For $\hat x\in M$,  $T\in \R$, and $y\in \p M$ we define the observation time function corresponding
to a point source at $(\hat x,T)\in M\times \R$,
\beq\label{def of cal T}
\mathcal T_{\hat x,T}(y)&=&\sup \{t\in \R;\hbox{ the set $\{y\}\times (-\infty,t)$ has a neighborhood}\\
& &\quad \quad \hbox{ $U\subset \p M\times \R$ such that $ G(\cdot,\cdot,\hat x,T)\big|_U =0$}\}. \nonumber
\eeq
In other words, $\mathcal T_{\hat x,T}(y)$ is the first time when the wave $G(\cdot,\cdot,\hat x,T)$  is observed on the boundary at the point $y$.

\begin{proposition} \label{prop: 1}
(i) For all $z\in {\p M}$, the pair $(\p M,g|_{\p M})$  and  the map $\Lambda$ determines function $\tau_{\p M}(z)$.

(ii) For all $z\in {\p M}$  and ${{\widehat t}}<\tau_{\p M}(z)$  the pair $(\p M,g|_{\p M})$  and the map $\Lambda$ determines  $\mathcal T_{x,T}(z)$ for  the point  $x=\gamma_{ z,\nu}({{\widehat t}})\in M$.

(iii) We have $\mathcal T_{\hat x,T}(y)=  \texttt{d}_M(y,\hat x)+T$. %
%
\end{proposition}

\begin{proof}  
Let us first prove (iii), and then (i) and (ii).

(iii) 
Using the finite velocity of the wave propagation for the wave equation, see  \cite{HormanderIV}, we obtain that the support of Green's function
$G(\cdot,\cdot,\hat x,T)$ is contained $J^+(q)\cap (\p M\times \R)$, where  $J^+(q)$ is in the causal future the point $q=(\hat x,T)\in M\times \R$, given by
\ba
J^+(q)=\{(y',s)\in M\times \R;\ s\ge  \texttt{d}_M(y',\hat x)+T\}.
\ea
This implies that $G(\cdot,\cdot,\hat x,T)=0$ in $J^+(q)$ and that $\mathcal T_{\hat x,T}(y)\geq  \texttt{d}_M(y,\hat x)+T$.
Next, we consider the opposite inequality. To this end, assume that there is $t_1>\texttt{d}_M(y,\hat x)+T$ such that $t_1< \mathcal T_{\hat x,T}(y)$. Then,  $G(\cdot,\cdot;\hat x,T)$ vanishes in an open set $U_1\subset \p M\times \R$ that contains  $\{y\}\times (-\infty,t_1)$. As $\partial_\nu G(\cdot, \cdot; \hat x,T)|_{\partial M\times \R}=0$, we then have that
the Cauchy data of $ G(\cdot, \cdot; \hat x,T)$ vanishes in the set $U_1$. Let $\psi_\varepsilon\in C^\infty(\R)$ be a function such that $\int_\R \psi_\varepsilon(t)dt=1$ and $\supp(\psi_\varepsilon)\subset (-\varepsilon,\varepsilon)$.
By the above, the function 
$$
 G^\varepsilon(x,t; \hat x,T)=\int_\R  G(x,t-t'; \hat x,T) \psi_\varepsilon(t') dt'
$$ is a $C^\infty$-smooth function satisfies the homogeneous wave equation
\begin{eqnarray}\label{GreensWaveEq 2}
  \left(\partial^2_t-\mathcal A\right) G^\varepsilon(\cdot, \cdot;\hat x,T)=0,\quad \hbox{on }(M\times \R)\setminus I_\varepsilon,\\
  G^\varepsilon(\cdot, \cdot; \hat x,T)|_{U^\varepsilon_1}=0;\,\,\partial_\nu G^\varepsilon(\cdot, \cdot; \hat x,T)|_{U^\varepsilon_1}=0\nonumber
\end{eqnarray}
where $U^\varepsilon_1\subset \partial M\times \R$ is a neigbhorhood of $ \{y\}\times (-\infty,t_1-\varepsilon)$ and $I_\varepsilon=\{\hat x\}\times 
 (T-\varepsilon,T+\varepsilon).$ Using Tataru's unique continuation theorem \cite{Ta1} 
 in the domain $M\times \R$ we see that 
$$
G^\varepsilon(x,t; \hat x,T)=0\quad\hbox{ for }(x,t)\in \{ (M\times \R)\setminus \{\hat x\}\times I_\varepsilon):\ t<t_1-\texttt{d}_M(x,y)-\varepsilon\}.$$
As $G^\varepsilon(x,t;\hat x,T)\to G(x,t; \hat x,T)$ in  the domain $(M\setminus \{x_0\})\times \R$ in sense  of distributions as $\varepsilon\to 0$, we see that  
$$
\hbox{$G(x,t; \hat x,T)=0$ for $(x,t)\in \mathcal V\setminus \{(\hat x,T)\}$},
$$
where $$\mathcal V=\{(x,t)\in M\times \R:\ t<t_1- \texttt{d}_M(x,y)\}.$$
Since $\mathcal V$ is an open neighborhood of the point $(\hat x,T)$, we see that
$G(\cdot, \cdot; \hat x,T)|_{\mathcal V}$ is a distribution supported in a single point $(\hat x,T)$.
By \cite{Rudin}, this implies that $F=G(\cdot, \cdot; \hat x,T)|_{\mathcal V}$ is finite sum of derivatives
of the delta distribution supported at $(\hat x,T)$.
 Considering such a distribution $F$ in local coordinates and computing its Fourier transform,
we see that $(\p_t^2-\mathcal A)F$ can not be the delta-distribution $\delta_{(\hat x,T)}(x,t)$. 
This is in contradiction with 
the equation (\ref{GreensWaveEq}), and hence we conclude that the claimed $t_1\in (\texttt{d}_M(y,\hat x)+T,\mathcal T_{\hat x,T}(y))$
can not exists. Thus $\mathcal T_{\hat x,T}(y)=  \texttt{d}_M(y,\hat x)+T$. This proves (iii).

(i) The map $\Lambda$ determines the functions $f_n{{(\alpha,\beta}},k)$. If
${{\widehat t}}>\tau_{\p M}(z)$, the limit (\ref{eq: A0B}) is zero. If
${{\widehat t}}<\tau_{\p M}(z)$, the considerations in the proof of claim (ii) show that the limit (\ref{eq: A0B}) is non-zero.
Thus $\Lambda$ determines $\tau_{\p M}(z)$.

(ii) The claim follows from the definition \eqref{def of cal T} of $\mathcal T_{\hat x,{{\widehat t}}}(y)$.

%
\end{proof}

%
%
%
%


By \eqref{prop: 1}  the pair $(\p M,g|_{\p M})$  and map $\Lambda$  determine the  function  $\tau_{\p M}(z)$ for all $z\in {\p M}$. Those  determine  also $\mathcal T_{x,T }(y)$ and $\texttt d_M(\hat x,y)$, $y\in\p M$ for the point  $x=\gamma_{ z,\nu}({{\widehat t}})\in M$
where  ${{\widehat t}}<\tau_{\p M}(z)$. As the distance function is continuous, we see that when ${{\widehat t}}\to t_1=\tau_{\p M}(z)$, we have that $\texttt d_M(\gamma_{ z,\nu}({{\widehat t}}),y)\to \texttt d_M(\gamma_{ z,\nu}(t_1),y)$. Thus   the pair $(\p M,g|_{\p M})$  and the map 
$\Lambda$ determine $\texttt d_M(x_0,y)$ for the point  $x=\gamma_{ z,\nu}({{\widehat t}})\in M$
for all  ${{\widehat t}}\le \tau_{\p M}(z)$ and $y\in \p M$. 
This implies that  the pair $(\p M,g|_{\p M})$  and $\Lambda$  determine the collection of boundary distance functions, that is, the set 
$$R(M)=\{\texttt d_M(\hat x,\,\cdotp)\in C(\p M):\ \hat x\in M\}.$$ Further, the set  $R(M)$ determines the isometry type of $(M,\texttt g)$, see \cite{KKL} (see also generalizations of this result in
 \cite{LaSak} (see also \cite{Ivanov2018}). Moreover, in the case when $ M\subset \R^n$ and $\A=-c(x)^2\Delta$ we can determine the Euclidean coordinates
 of the point $\hat x=\gamma_{z,\nu}(\hat t)$ using Cor. \ref{cor: focusing coordinates}. Hence we can determine
  vector $\hat v=\lim_{t\to \hat t-}\p_t\gamma_{z,\nu}(t)$ and $c(\hat x)=1/\|\hat v\|_{\R^n}$. This gives an algorithm to determine the unknown wave speed $c(x)$ at all points 
  $x\in M$.
 \medskip

\noindent {\bf Acknowledgements:}
The research has been partially supported by 
EPSRC EP/D065711/1, and 
 Academy of Finland, grants 273979, 284715, 312110.

\appendix

\section{}

\label{appA}
In this appendix, we show that the Blagovestchenskii identities (\ref{blago1}),(\ref{Tformula2}), and the energy identity (\ref{Tformula11}) hold. 

\subsection{Proof of the Blagovestchenskii identity 1 (\ref{blago1})}
\label{appA1}
We have following version of the \emph{Blagovestchenskii identity}
\begin{align*}
\int_{M} u^f(T)u^h(T)\,\mathrm{dV} =\int_{{\p M}\times [0,2T]}
(Kf)(x,t) h(x,t)\, dS(x) dt, 
\end{align*}
where $f,h\in V$.
The proof given here is in a slightly different context that the one done e.g.\ in \cite{KKL}.
\begin {proof} For boundary value problem 
\begin{align}
\label{eq: Wave12}
\begin{cases}
  \p_t^2u^f(x,t)-\mathcal A u^f(x,t)=0,\quad \hbox{ in } M\times \R_+,\\
  u^f|_{t=0}=0,\quad u^f_t|_{t=0}=0,  \\
  \p_\nu u^f|_{\p M \times \R_+}=f, 
\end{cases}
\end{align}
let us assume that we have solutions $u^f$ and $u^h$ with respect to boundary sources $f,h\in V)$.
Let us define
\beq
\nonumber
w(t,s)=\int_{M} u^f(t) {u^h(s)}\,\mathrm{dV}_\mu.
\eeq
Integrating by parts, we see that
\beq
\label{right hand side of wave eq}& &
(\p^2_t-\p^2_s)w(t,s)
\hspace{-5em}\\
&=& \nonumber
-\int_{M} \big[\A u^f(t) {u^h(s)}- u^f(t)
{\A u^h(s)}\big]\,\mathrm{dV}_\mu(x)
\\ \nonumber
& &-\int_{\p M} \big[\p_\nu u^f(t)u^h(s)- u^f(t)
{\p_\nu u^h(s)}\big]\,dS_{\texttt g}
\\ \nonumber
&=&\int_{\p M} \big[(-\p_\nu u^f(t)+\eta u^f(t)) u^h(s)- u^f(t)
(-\p_\nu u^h(s)+\eta u^h(s))\big]\,dS_{\texttt g}
\\ \nonumber
&=& \int_{\p M} \big[f(t) \HL h(s)-
\HL f(t)
{h(s)}\big]\,dS_{\texttt g}.
\eeq
Moreover, as
\bfo
\left. w \right| _{t=0}=\left. w\right| _{s=0}=0, \quad
\left. \p _t w\right| _{t=0}= \left. \p_s w\right| _{s=0}=0,
\efo
we can consider (\ref{right hand side of wave eq}) as one dimensional wave equation with known right hand side and vanishing initial and boundary data. Solving this initial-boundary value problem,  we obtain 
\beq
\int_M u^f(x,T) {u^h(x,T)}\,\mathrm{dV}_\mu(x)=
\label{4.60}
\eeq
\bfo
\int_{[0,2T]^2}  \int_{\p M} J(t,s) \big[f(t)  {(\Lambda h)(s)}-
(\Lambda f)(t)  {h(s)}\big]\,dS_{\texttt g}(x)dtds,
\efo
where $J$ is as defined in (\ref{J-operator}).
\smallskip

The Schwartz kernel of $\Lambda$ is the Dirichlet boundary value of the  Green's function $G(x,x',t-t')$ satisfying
\begin{align}
\label{eq: Wave green}
( \p_t^2+\A)G_{x',t'}(x,t) &= \delta_{x'}(x)\delta(t-t')\quad \hbox{ in
} M\times \R_+,\\
G_{x',t'}|_{t=0} = 0, &\quad \p_tG_{x',t'}|_{t=0}=0,  \quad
B_{\nu,\eta} G_{x',t'}|_{\p M\times \R_+} = 0, \nonumber
\end{align}
where $G_{x',t'}(x,t)=G(x,x',t-t')$.
As
\ba
G(x,x',t-t')=G(x',x,t-t'),
\ea
we see that
$
\Lambda^*=R \Lambda R
$
where $Rf(x,t)=f(x,2T-t)$ is the time reversal map.

Thus,  we can rewrite  formula (\ref{4.60}) in the form
\beq\label{kaava A}\quad\quad
\int_M u^f(x, T) {u^h(x, T)}\,\mathrm{dV}_\mu(x)
=\int_{\p M\times [0,2T]} (Kf)(x,t)\,h(x,t)\,dS_{\texttt g}(x)dt\hspace{-1cm}
\eeq 
where $K$ is defined in \pef{K-operator}.

Analyzing (\ref{K-operator}), we see that the inner product in the
{\newtextt left}-hand side of (\ref{kaava A}) can be found by making two measurements, one with the input $f$ and the other with the input  $RJf,$
obtained from $f$ by  basic operations of the time reversal $R$ and the
time filtering $J$.
\end {proof}
\subsection{Proof of the Blagovestchenskii identity 2 (\ref{Tformula2})}
\label{appA2}
We have following version of the \emph{Blagovestchenskii identity (\ref{Tformula2})}
\begin{align*}
 \bra u^h(T),1 \cet_{L^2(M)}=-\bra h,\Phi _T \cet_V,
\end{align*}
where $h\in V$ and $\Phi _T$ is as defined in (\ref{ihk}).
\begin {proof}Let us assume that we have solution $u^h$ for problem (\ref{eq: Wave12}) with respect to boundary source $h\in V.$
Let us define
\ba
I(t)&=&\int_{M}\left[u^h(x,t)\right]\,\mathrm{dV}_\texttt{g}(x).
\ea
Differentiation respect the time gives us
\ba
\p_t^2I(t)&=&\int_{M}\left[\p_t^2u^h(x,t)\cdot 1\right]\,\mathrm{dV}_\texttt{g}(x).
\ea
Using the definition of problem (\ref{eq: Wave12}) we have
\ba
\p_t^2I(t)&=&\int_{M}\left[\mathcal A u^h(x,t)\cdot 1\right]\,\mathrm{dV}_{\texttt g}(x).
\ea
Integrating by parts, we see that
\ba
\p_t^2I(t)&=&
\int_M \big[u^h(x,t)\mathcal A 1\big]\,\mathrm{dV}_\mu(x)
-\int_{\p M} \big[\p_\nu u^h(x,t)\cdot 1-u^h(x,t)\p_\nu 1\big]\,dS_{\texttt g}(x)
\ea
Thus
\ba
\p_t^2I(t)&=&-\int_{\p M} \big[h(x,t)\big]\,dS_{\texttt g}(x)
\ea
At time $t=0$ we assumed that initial values $\p_t u(x,0)=0$ and $u(x,0)=0$. Thus we have $I(0)=0$ and $\p_tI(0)=0$. Using this we get 
\ba
I(t)&=&-\int_0^t\int_0^{t'}\int_{\p M} h(x,t'')\,dS_{\texttt g}(x)dt''dt'.
\ea
Let us define
\ba
J_t&=&\big\{(t',t'')|0\le t''\le t'\le t\big \}.
\ea
Using the indicator function we get
\ba
I(t)&=&-\int_0^t\int_0^{t}\int_{\p M} h(x,t'')\mathds{1}_{J_t}(t',t'')\,dS_{\texttt g}(x)dt''dt'.
  \ea
Then we chance the order of integration gives us
\ba
I(t)&=&-\int_{\p M} \int_0^t h(x,t'') \big[\int_0^{t} \mathds{1}_{J_t}(t',t'')dt'\big]dt''\,dS_{\texttt g}(x)\\
&=&-\int_{\p M} \int_0^t h(x,t'') \big[\int_{t''}^{t}1 dt'\big]dt''\,dS_{\texttt g}(x)\\
&=&-\int_{\p M} \int_0^t h(x,t'')\big[t-t''\big]dt''\,dS_{\texttt g}(x).
\ea
For $t=T$ we have
\ba
I(T)&=&-\int_{\p \M} \int_0^T h(x,t'')\big[T-t''\big]dt''\,dS_{\texttt g}(x).
\ea
Thus 
\ba
I(T)&=&-\int_{\p \M} \int_0^{2T} h(x,t'')\hat P\big[T-t''\big]dt''\,dS_{\texttt g}(x)=-\bra h,\Phi _T\cet_V.
\ea
\end{proof}

\bibliographystyle{abbrv}
\bibliography{references}

\end{document}